\documentclass[11pt,fullpage]{article}
\usepackage{amsmath}
\usepackage{amsfonts,amssymb}
\usepackage{color}	
\usepackage{amsthm}
\usepackage{graphics,graphicx}
\usepackage{hyperref}
\usepackage[right = 2.5cm, left=2.5cm, top = 2.5cm, bottom =2.5cm]{geometry}
\newcommand{\Real}{\mathbb R}
\def\bR{\mathbb{R}}

\def\BMO{\text{BMO}}
\def\lec{\lesssim}
\def\gec{\gtrsim}
\def\dist{\text{dist}}

\newcommand{\eqn}[1]{\eqref{e:#1}}
\newcommand{\ps}[1]{\left(#1\right)}
\newcommand{\norm}[1]{\|#1\|}
\newcommand{\abs}[1]{\left\vert#1\right\vert}
\newcommand{\set}[1]{\left\{#1\right\}}

\newcommand{\grad}{\nabla}

\newcommand{\K}{\mathcal{K}}

\newcommand{\Naturals}{\mathbb N}

\def\d{\partial}

\newcommand{\avint}[1]{\mathchoice
{\mathop{\vrule width 6pt height 3 pt depth -2.5pt
\kern -8.8pt \intop}\nolimits_{#1}}%
{\mathop{\vrule width 5pt height 3 pt depth -2.6pt
\kern -6.5pt \intop}\nolimits_{#1}}%
{\mathop{\vrule width 5pt height 3 pt depth -2.6pt
\kern -6pt \intop}\nolimits_{#1}}%
{\mathop{\vrule width 5pt height 3 pt depth -2.6pt \kern -6pt
\intop}\nolimits_{#1}}}

\newtheorem{theorem}{Theorem}
\theoremstyle{definition}
\newtheorem{remark}{Remark}

\theoremstyle{lemma}

\theoremstyle{definition}
\newtheorem{definition}{Definition}
\theoremstyle{lemma}
\newtheorem{lemma}{Lemma}

\begin{document}

\title{Bounded Mean Oscillation and the Uniqueness of Active Scalar Equations}

\author{Jonas Azzam\footnote{\textit{jonasazzam@math.washington.edu}, University of Washington-Seattle} \, and Jacob Bedrossian\footnote{\textit{jacob@cims.nyu.edu}, New York University, Courant Institute of Mathematical Sciences. Partially supported by NSF Postdoctoral Fellowship in Mathematical Sciences, DMS-1103765}}

\date{}

\maketitle

\begin{abstract} 
We consider a number of uniqueness questions for several wide classes of active scalar equations, unifying and generalizing the techniques of several authors \cite{Yudovich63,Yudovich95,BertozziSlepcev10,BertozziBrandman10,BRB10,Rusin11}.
As special cases of our results, we provide a significantly simplified proof to the known uniqueness result for the 2D Euler equations in $L^1 \cap BMO$ \cite{Vishik99} and provide a mild improvement to the recent results of Rusin \cite{Rusin11} for the 2D inviscid surface quasi-geostrophic (SQG) equations, which are now to our knowledge, the best results known for this model.
We also obtain what are (to our knowledge) the strongest known uniqueness results for the Patlak-Keller-Segel models with nonlinear diffusion \cite{Blanchet09,BRB10}. 
 We obtain these results via technical refinements of energy methods which are well-known in the $L^2$ setting but are less well-known in the $\dot{H}^{-1}$ setting. 
The $\dot{H}^{-1}$ method can be considered a generalization of Yudovich's classical method \cite{Yudovich63,Yudovich95,MajdaBertozzi} and is naturally applied to equations such as the Patlak-Keller-Segel models with nonlinear diffusion, and other variants \cite{BertozziBrandman10,BertozziSlepcev10,BRB10}. 
Important points of our analysis are an $L^p$-$BMO$ interpolation lemma and a Sobolev embedding lemma which shows that velocity fields $v$ with $\grad v \in BMO$ are locally log-Lipschitz; the latter is known in harmonic analysis but does not seem to have been connected to this setting. 
\end{abstract}

\section{Introduction} 
The term active scalar refers to a wide variety of nonlinear advection-diffusion equations for a scalar density which, through a nonlocal operator such as an elliptic PDE, gives rise to the advective velocity field (to the authors' knowledge, the terminology is originally due to P. Constantin \cite{Constantin94}, although we are using it for a wider class). 
The purpose of this paper is to study the uniqueness of sufficiently regular weak solutions to such equations in $\Real^d$ for dimensions $d \geq 2$. 
Our main interest will be achieving uniqueness at relatively low regularity and we will not be overly concerned with how quickly our solutions decay at infinity. 
The methods can also be extended to more general domains, for example, bounded Lipschitz domains, as elaborated further in \S\ref{sec:BoundedDomains}.   
 
The best known active scalar equations are the 2D Euler equations ($\nu = 0$) and Navier-Stokes equations ($\nu > 0$) in vorticity form 
\begin{equation*} 
\left\{
\begin{array}{l}
  \omega_t + v\cdot \grad \omega = \nu \Delta \omega, \\
  v = \grad^{\perp}(\Delta)^{-1} \omega. 
\end{array}
\right.
\end{equation*}
The classical proof of uniqueness due to Yudovich \cite{Yudovich63} proves that weak solutions are unique provided 
$\omega(t) \in L^\infty(0,T;L^1(\Real^2)) \cap L^\infty(0,T;L^\infty(\Real^2))$. 
He also developed the classical existence theory for solutions to 2D Euler of this form based on the method characteristics (see also \cite{MajdaBertozzi}).  
Given two weak solutions $\omega_1,\omega_2$ his method estimated the evolution of $\norm{v_1(t) - v_2(t)}_2$ (with the obvious notation).
Measuring $\norm{\omega_1 - \omega_2}_{\dot{H}^{-1}}$, however, provides a natural generalization of Yudovich's method to many other active scalars of interest (see \cite{Robert97,BertozziSlepcev10,BertozziBrandman10,BRB10} and below). 
Later, Yudovich extended his proof to cover cases with unbounded vorticity \cite{Yudovich95}, but with certain restrictions on the growth of $\norm{\omega}_p$ as $p \rightarrow \infty$ which ruled out logarithmic singularities.
In that article, Yudovich also seems to have been the first to note the connection between the Osgood condition for the uniqueness of ODEs and the uniqueness of the 2D Euler equations. 
Vishik \cite{Vishik99} later proved what is to our knowledge the optimal known uniqueness results which include the class $\omega(t) \in L^\infty(0,T;L^1(\Real^2)) \cap L^\infty(0,T;BMO(\Real^2))$. 
His methods, in which of course the Osgood condition plays a role, are based on wavelet decompositions and a precise Littlewood-Paley decomposition of solutions to the 2D Euler equations in standard Eulerian form. 
One of the several results of our article is that Yudovich's significantly simpler method can also be used to prove uniqueness to the 2D Euler equations with $\omega \in L^1 \cap BMO$. 
We emphasize here that we only prove uniqueness in this class, not the more difficult question of existence. 
Ultimately, the reason this is possible rests on the facts that Calder\'on-Zygmund-type singular integral operators map $BMO$ to itself \cite{BigStein} and vector fields $v$ with $\grad v \in BMO$ are locally log-Lipschitz (Lemma \ref{lem:BMO_logLip} below), and hence satisfy the Osgood condition.

The first class of active scalars we study are PDE of the form, 
\begin{equation}
\rho_t + \grad \cdot (\rho V\rho) = \Delta A(\rho), \label{def:Type1} 
\end{equation} 
where $A$ is non-decreasing (potentially zero) and $V$ is a linear operator which is roughly smoothing of order one (e.g. derivatives of solutions to second order elliptic PDEs, $S^{-1}$ pseudo-differential operators etc).
As mentioned above, the most famous of such equations are of course the 2D Euler and Navier-Stokes equations. 
Other examples include the parabolic-elliptic Patlak-Keller-Segel model and related variants \cite{BlanchetEJDE06,Blanchet09,BertozziSlepcev10,BRB10} where one would have for example $V\rho = -\grad(\Delta)^{-1}\rho$, as well as the inviscid aggregation equations \cite{Laurent07,BertozziLaurentRosado10,BertozziBrandman10}. 
As alluded to above, we will essentially apply Yudovich's energy method to these equations, measuring the difference of weak solutions in $\dot{H}^{-1}$. 
This also has the very useful additional benefit of treating nonlinear 
filtration-type diffusion very naturally, which is well-known to be a monotone operator in $H^{-1}$ \cite{VazquezPME}.   
Indeed, the proof applies regardless of how pathological or strongly degenerate the nonlinear diffusion is. 
This technique was applied in \cite{BertozziSlepcev10,BRB10} to Patlak-Keller-Segel models and aggregation equations with porous media type diffusion. 

A second class of active scalar we consider are essentially variants of the 2D surface quasi-geostrophic equations studied by, for example \cite{ConstantinCordobaWu01,Wu05,CaffarelliVasseur06,KiselevNazarov07,Kiselev10,Rusin11},
\begin{equation}
\rho_t + \grad \rho \cdot V\rho = -\nu(-\Delta)^{\gamma}\rho, \label{def:Type2}  
\end{equation}
with $\nu \geq 0$, $\gamma \in (0,1]$ and $V$ is a linear operator of roughly Calder\'on-Zygmund type which is divergence free, i.e. $\grad \cdot V\rho = 0$. 
The 2D SQG equations are the special case given by $V\rho = \grad^\perp (-\Delta)^{-1/2}\rho = (-R_2\rho, R_1\rho)$, where $R_j$ are the standard Riesz transforms. 
To study these equations we measure the difference of weak solutions in $L^2$, but largely we use similar methods and techniques as used for the class \eqref{def:Type1}. Uniqueness results for both \eqref{def:Type1} and \eqref{def:Type2} are discussed below in \S\ref{sec:Type1and2}.  

For inviscid active scalar equations, the authors are not aware of any uniqueness results which extend to problems with velocity fields less regular than those that satisfy the Osgood condition. 
Quantitative stability estimates, such as those here and those in \cite{LoeperVP06,LoeperSG06,Rusin11,Vishik99}, transfer the regularity of the velocity field into stability in time of the norm being measured. 
In the Eulerian framework of an $\dot{H}^{-1}$ estimate here, or in the work of Vishik \cite{Vishik99}, the intuition for this is not immediately apparent, as the modulus of continuity of the velocity field is not explicitly used in the proof. 
However, the Lagrangian methods of Loeper \cite{LoeperVP06,LoeperSG06,CarrilloRosado10,BertozziLaurentRosado10}, which measure stability in the Euclidean Wasserstein distance, make direct use of the Osgood condition as one which ensures that characteristics are well-posed and hence the induced flow map has sufficient stability and regularity properties. 
In contrast, non-uniqueness results for the 2D Euler equations \cite{Scheffer93,Shnirelman97,DeLellisSzekelyhidi08} only apply to problems of far lower regularity than those which satisfy the Osgood condition, leaving a rather wide gap between the cases which are settled.   
The DiPerna-Lions theory of renormalized solutions provides uniqueness for linear transport equations under very weak hypotheses, however, perhaps not coincidentally, these solutions do not seem to have well understood stability properties \cite{DiPernaLions89,AmbrosioIM04}.

For viscous problems, it is well-known that linear diffusion should improve the picture by providing a potentially useful negative term in stability estimates. 
In \S\ref{sec:Dissipation} we show that this can be applied also in the $\dot{H}^{-1}$ method for models of the form \eqref{def:Type1} with $A(\rho)= \rho$. 
As a special case, our results show that $\int_0^t \int \abs{\rho}^2 + \abs{\grad \rho}^2 dx dt < \infty$ is sufficient for uniqueness of weak solutions to \eqref{def:Type1}.

Some of the simplicity of the method arises from the use of scale-invariant (homogeneous) norms: for the first class we estimate stability in $\dot{H}^{-1}$, for the second class, we estimate stability in $L^2$. 
The use of more norms which could treat specific length-scales with more precision, such as Besov-type norms as in \cite{Vishik99}, likely allows one to loosen requirements on the integrability of solutions at infinity and also treat active scalars such that the velocity operator $V$ is of a type in between those above, for example as studied in \cite{ConstantinIyerWu08}. 
All of this, of course, would complicate the proof significantly.

Turning back to the 2D Euler equations as the canonical example, we point out here that the \emph{regularity} implied by $\omega \in BMO$, as opposed to the \emph{integrability}, is very crucial. Lemma \ref{lem:BMO_linear_p} below shows that $\omega \in L^1\cap BMO$ imparts a certain amount of control on the integrability, that is, $\norm{\omega}_p \lesssim p$ as $p \rightarrow \infty$. However, we emphasize that this integrability alone is not enough for the following proof to work. 
Yudovich's proof in \cite{Yudovich95} uses only integrability and requires a stronger restriction on how fast $\norm{\omega}_p$ can be allowed to grow in $p$.
The results in \cite{Kelliher11} indicate that initial data which satisfies Yudovich's restrictions (but are not $BMO$) can result in velocity fields which are not log-Lipschitz (but still Osgood), and hence are in fact less regular than those studied here. 
Hence, it is important to emphasize that simply having logarithmic singularities in the density is not the key idea.   
The regularity conveyed by $BMO$ is used twice in our work: that Calder\'on-Zygmund-type singular integral operators map $BMO$ to itself and that velocity fields with $\grad v \in BMO$ are locally log-Lipschitz (Lemma \ref{lem:BMO_logLip}).

\subsubsection*{Notation and Conventions}
We briefly mention the definitions of weak solution we are using. Our theorems will generally require stronger regularity assumptions than these definitions; these simply provide the baseline. 
We use the following definition of weak solution to \eqref{def:Type1}. 
This notion is stronger than a standard distribution solution, as we must be able to measure the stability of solutions $\dot{H}^{-1}$, which will require taking test functions in $\dot{H}^{1}$. In 2D, a slight modification (suggested in \cite{BertozziBrandman10}) must be made as $\dot{H}^1(\Real^2)$ is not a sensible function space.
 
\begin{definition}[Weak Solution to \eqref{def:Type1}]
If $d\geq 3$ let $\mathcal{V} = \dot{H}^1$ (hence $\mathcal{V}^\star = \dot{H}^{-1}$) and if $d = 2$ then let $\mathcal{V} = \set{f \in L^\infty : \grad f \in L^2}$. 
We say a measurable function $\rho:[0,T]\times \Real^d \rightarrow \Real$ is a weak solution on some time interval $[0,T]$ to \eqref{def:Type1} if $\rho_t \in L^2(0,T;\mathcal{V}^{\star}(\Real^d))$, $\grad A(\rho) \in L^2(0,T;L^2(\Real^d))$, $\rho V\rho \in L^2(0,T;L^2(\Real^d))$ and for all $\phi \in L^2(0,T;\mathcal{V})$ we have, 
\begin{equation*}
\int_0^T <\rho_t,\phi(t)> dt = -\int_0^T \int \left(\grad A(\rho(t)) - \rho(t)V\rho(t)\right) \grad \phi(t) dt. 
\end{equation*}
\end{definition}
For \eqref{def:Type2} we may take the standard notion of distribution solution as we need only to measure stability in $L^2$. 
\begin{definition}[Weak Solution to \eqref{def:Type1}]
We say a measurable function $\rho:[0,T]\times \Real^d \rightarrow \Real$ is a weak solution on some time interval $[0,T]$ to \eqref{def:Type2} if $\rho \in L^2(0,T;L^2)$ and for all $\phi \in C^\infty_c((0,T) \times \Real^d)$ we have, 
\begin{equation*}
\int_0^T \int \rho\left(\phi_t - V\rho \cdot \grad \phi + \nu(-\Delta)^\gamma \phi\right) dx dt = 0.   
\end{equation*}
\end{definition}
We denote the set $\set{u > k} := \set{x\in D: u(x) > k}$, if $S \subset \Real^d$ then $\abs{S}$ denotes the Lebesgue measure and $\mathbf{1}_{S}$ denotes the standard characteristic function.
We denote 
\[f_{A} := \avint{A} f dx=\frac{1}{|A|}\int_{A} f.\] 
We denote the standard Lebesgue spaces $\norm{f}_p = \left(\int \abs{f}^p dx \right)^{1/p}$.
We take the unitary convention of the Fourier transform, 
\begin{equation*}
\hat{f}(\xi) = \frac{1}{(2\pi)^{d/2}}\int f(x)e^{-ix\cdot \xi}d\xi, 
\end{equation*}
and denote the homogeneous Sobolev spaces $\dot{H}^s$ as the closure of the Schwartz space under the norm $\norm{f}_{\dot{H}^{s}} := \norm{\abs{\xi}^s \hat{f}}_{L^2_\xi}$ and the inhomogeneous Sobolev spaces $H^s$ as the closure of the Schwartz space under the norm $\norm{f}_{H^s} := \norm{(1 + \abs{\xi}^2)^{s/2}\hat{f}}_{L^2_\xi}$. 

We use $\mathcal{N}$ to denote the Newtonian potential: 
\begin{equation*}
\mathcal{N}(x) = 
\left\{
\begin{array}{ll}
  \frac{1}{2\pi}\log \abs{x} & d = 2 \\ 
  \frac{\Gamma(d/2 + 1)}{d(d-2)\pi^{d/2}}\abs{x-y}^{2-d} & d \geq 3. 
\end{array}
\right.
\end{equation*}
We will denote space-time norms with the short hand 
\begin{equation*}
\norm{f}_{L^p_tL^q_x} = \norm{f}_{L_t^p(0,T;L_x^q(\Real^d))}, 
\end{equation*}
as the time interval and domain will usually be obvious from context. 
In formulas we use the notation $C(p,k,M,..)$ to denote a generic constant, which may be different from line to line or even term to term in the same computation. 
In general, these constants will depend on more parameters than those listed, for instance those associated with the problem such as $\K$ and the dimension but these
dependencies are suppressed.
We use the notation $f \lesssim_{p,k,...} g$ to denote $f \leq C(p,k,..)g$ where again, dependencies that are not relevant are suppressed. We will also frequently work with dyadic cubes, 
\[\bigcup_{n\in\mathbb{Z}} \left\{ \prod_{i=1}^{d}[a_{i},a_{i}+2^{n}]: a_{i}\in 2^{n}\mathbb{Z}\right\}\]
and the side length of such a cube $Q$ in this collection will be denoted $\ell(Q)$. For $\lambda>0$, we write $\lambda Q$ for the cube of the same center as $Q$ but $\lambda$ times the radius.

\section{Lemmas} \label{sec:Lems}
We begin with a few relatively short lemmas which, while simple, provide the key technical tools for this work. 
The first lemma is essentially a logarithmic variant of Morrey's inequality, allowing one to assert the log-Lipschitz regularity of a vector field based on the gradient having bounded mean oscillation. 
Although, as discussed above, we do not need to make explicit use of this lemma, it reflects probably the most fundamental observation, as this regularity is ultimately connected to the stability necessary to prove the uniqueness theorems below.    
It also shows that for inviscid problems, the Lagrangian proof of Loeper \cite{LoeperVP06} could likely be extended to cover all of the (inviscid) cases considered in Theorem \ref{thm:Hdot1Inviscid} below as well as the Vlasov-Poisson equation (one would only need to extend Theorem 2.9 in \cite{LoeperVP06}). 

\begin{lemma} \label{lem:BMO_logLip}
Let $v \in L^p(\Real^d)$ for some $1 \leq p \leq \infty$ such that $\grad v \in BMO(\Real^d)$ and 
\begin{equation*}
\sup_{x \in \Real^d}\int_{B(x,1)} \abs{\grad v} dz < \infty.  
\end{equation*}
Then $v \in L^\infty$ and for $\abs{x-y}$ sufficiently small (depending on $\sup_{x \in \Real^d}\int_{B(x,1)} \abs{\grad v} dz$ but not on $x,y$), 
\begin{equation*}
\abs{v(x) - v(y)} \lesssim \abs{x-y}\abs{\log\abs{x-y}}\norm{\grad v}_{BMO}.
\end{equation*}
\end{lemma}  

This lemma is known in harmonic analysis (c.f. \cite[Theorem A.2, Proposition A.3]{mcmullen1996renormalization}). In the aforementioned reference, for example, it is used to establish that quasiconformal vector fields are log-Lipschitz. A standard proof involves first showing that a function satisfying the conditions of the lemma is in the Zygmund class, whose functions are log-Lipschitz continuous. For the reader's convenience, however, we supply a self-contained proof in the appendix.\\

The John-Nirenberg inequality \cite{BigStein} is one classical way to quantify the idea that functions in $BMO$ have at most logarithmic singularities and in particular, asserts that any $f \in BMO$ is in $L_{loc}^p$. 
The following interpolation-type lemma provides a slight variant of this idea which will prove to be a simple but very important technical tool.
In order to treat bounded domains, we will also prove the lemma in the natural class of so-called uniform domains. 
\begin{definition} 
A domain $\Omega$ is a {\it C-uniform domain} if, for any points $x,y\in \Omega$, there is a curve $\gamma\subseteq \Omega$ connecting $x$ and $y$ with $\text{length}(\gamma)\leq C|x-y|$ and such that for all $z\in \gamma$, $B(z,\frac{\min\{|x-z|,|y-z|\}}{C})\subseteq \Omega.$
\label{d:uniform-domain}
\end{definition}
There are several different equivalent definitions of uniform domains in the literature (see \cite{Vaisala88}), but this one will suffice for our purposes. Visually, a uniform domain is one where any pair of points $x$ and $y$ have a crescent shape\footnote{A ``cigar" shape is the more common albeit misleading terminology in the literature.} of bounded eccentricity contained in $\Omega$ whose corners are $x$ and $y$. Moreover, it is straightforward to confirm that  Lipschitz domains are uniform. The class of $C$-uniform domains is natural here due to the extension theorem of Peter Jones \cite{Jones80}, which we will exploit in proving the following lemma. 

\begin{lemma} \label{lem:BMO_linear_p}
Let $f \in L^{p_0} \cap BMO(\Omega)$, where $\Omega$ is a $C$-uniform domain, $p_0 < \infty$. Then, for $p$, $p_0 < p < \infty$, 
\begin{equation}
||f||_{p} \lec_{d,p_{0}} p^{1-\frac{p_{0}}{p}} ||f||_{BMO}^{1-\frac{p_{0}}{p}}||f||_{p_{0}}^{\frac{p_{0}}{p}}.
\label{e:p<bmop0}
\end{equation}
In particular, for large $p$, the $L^p$ norm grows at most linearly as $p \rightarrow \infty$. 
\end{lemma}
\begin{proof}
We first prove the lemma in the case of $\Omega=\Real^d$ where the proof proceeds naturally. 
We will use the Calder\'on-Zygmund decomposition which splits $f$ into a sum of two parts: a ``good" part that is uniformly bounded, which we will bound using the $p_{0}$-norm of $f$,  and a ``bad" part which, though not bounded, has mean zero locally and may be bounded by the BMO norm of $f$ (see \cite[I.4]{BigStein}).

Without loss of generality, we may assume $f$ has compact support. Define the dyadic maximal function 
\begin{equation*}
Mf(x) = \sup_{x \in Q} \frac{1}{\abs{Q}}\int_Q \abs{f}(y) dy, 
\end{equation*}
where the supremum is taken over all dyadic cubes $Q$ containing $x$.

Let  $\alpha>0$ to be chosen later, $E_{\alpha}=\{Mf>\alpha\}$, $f_{1}=f\chi_{E_\alpha^{c}}$, $f_{2}=f-f_{1}$. 
Note by the Lebesgue differentiation theorem $f_1 \leq \alpha$. 
Then
\begin{equation}
||f_{1}||_{p}^{p} \leq \alpha^{p-p_{0}}||f||_{p_{0}}^{p_{0}}.\label{ineq:f_1}
\end{equation}
Let $\{Q_{k}\}$ be the collection of maximal dyadic cubes with disjoint interiors such that
\[\avint{Q_{k}}|f|>\alpha.\]
Each $x\in E_{\alpha}$ is certainly contained in such a cube $Q$ by the definition of $E_{\alpha}$, and that cube must necessarily be contained in $E_{\alpha}$,  for otherwise, there is $y\in Q\cap E_{\alpha}^{c}$, and by definition of $Mf(y)$, $Mf(y) > \alpha$, a contradiction.
This maximality also implies 
\[\avint{Q_{k}}|f|\lesssim_{d} \alpha.\]
Since $f$ is in BMO, we also have (see \cite[IV.1.3 (13)]{BigStein}) 
\[\avint{Q_{k}} |f-f_{Q_{k}}|^{p}\lesssim_{d} \left(p\norm{f}_{BMO}\right)^{p}.\]
We now estimate
\begin{align*}
\int|f_{2}|^{p} 
& = \sum_{k}\int_{Q_{k}}|f|^{p}\\ & 
\leq 2^{p}\sum_{k}\left( \int_{Q_{k}}|f-f_{Q_k}|^{p}+\int_{Q_{k}}|f_{Q_{k}}|^{p}\right) \\ & 
 \lesssim_{d} 2^{p}\sum_{k}\left(p\norm{f}_{BMO})^{p}|Q_{k}|+\alpha^{p}|Q_{k}|\right) \\ & 
 =2^{p}\left( (p\norm{f}_{BMO})^{p}+\alpha^{p})|\{Mf>\alpha\}|\right) \\ & 
 \lesssim_{p_0,d} 2^{p}((p\norm{f}_{BMO})^{p}+\alpha^{p})\frac{||f||_{p_{0}}^{p_{0}}}{\alpha^{p_{0}}},
\end{align*}
where in the last line we used the Chebichev and Hardy-Littlewood maximal inequalities in showing
\[|\{Mf>\alpha\}|\leq \frac{||Mf||_{p_{0}}^{p_{0}}}{\alpha^{p_{0}}}\lesssim_{p_0,d}\frac{||f||_{p_{0}}^{p_{0}}}{\alpha^{p_{0}}}.\]
We now let $\alpha=||f||_{p_{0}}$ and obtain
\begin{equation} \int |f_{2}|^{p} \lesssim_{d} 2^{p} ((p\norm{f}_{BMO})^{p}+||f||_{p_{0}}^{p})
\label{ineq:f_{2}} 
\end{equation}
Note that \eqref{ineq:f_1} implies $||f_{1}||_{p}^{p}\leq ||f||_{p_{0}}^{p}$. Adding this estimate to \eqref{ineq:f_{2}} and taking $p$th-roots of both sides gives
\begin{equation}
\norm{f}_p \lesssim_{d,p_{0}} p\norm{f}_{BMO} + \norm{f}_{p_0}
\label{e:bmo+p0}
\end{equation}
Now, since this holds for all functions $f\in BMO\cap L^{p_{0}}$ and this class is scale and dilation invariant, we can scale the argument of $f$ by a factor $\delta$ and the inequality will still hold. Let $f_{\delta}(x)=f(\delta x)$. The effect on the norms is
\[||f_{\delta}||_{p}=\delta^{-\frac{d}{p}}||f||_{p}, \;\; ||f_{\delta}||_{p_{0}}=\delta^{-\frac{d}{p_{0}}}||f||_{p_{0}}, \;\; \mbox{ and } \;\; ||f_{\delta}||_{BMO}=||f||_{BMO}.\]
Hence, letting $\eta=\delta^{-\frac{1}{d}}$, we have that for all $\eta$, 
\[\eta^{\frac{1}{p}} \norm{f}_p \lesssim_{d,p_{0}} p\norm{f}_{BMO} + \eta^{\frac{1}{p_{0}}} \norm{f}_{p_0}.\]
Choosing $\eta$ such that
\[\eta=\ps{\frac{p||f||_{BMO}}{||f||_{p_{0}}}}^{p_{0}},\]
implies, 
\begin{align*}
\norm{f}_p  & 
\lesssim_{d,p_{0}} \eta^{-\frac{1}{p}}(p\norm{f}_{BMO} + \eta^{\frac{1}{p_{0}}} \norm{f}_{p_0})\\ & 
 = 2\eta^{-\frac{1}{p}}p||f||_{BMO} \\ & 
 =2\ps{\frac{p||f||_{BMO}}{||f||_{p_{0}}}}^{-\frac{p_{0}}{p}} p||f||_{BMO} \\ & 
 =2p^{1-\frac{p_{0}}{p}} ||f||_{BMO}^{1-\frac{p_{0}}{p}}||f||_{p_{0}}^{\frac{p_{0}}{p}}.
\end{align*}

We now begin the process of generalizing the lemma to $C$-uniform domains, but first focus on the case of {\it unbounded domains}. The proof of $\Real^d$ does not translate immediately to this setting; if we pick $\alpha>0$ and maximal cubes $Q_{j}$ in $\Omega$ as before, the cubes may be maximal merely because their parent cubes are not contained in $\Omega$, and hence we cannot guarantee $\avint{Q_{j}}|f|\sim \alpha$. To overcome this issue, we employ an extension theorem of Peter Jones to reduce to the $\bR^{d}$ case.
\begin{theorem}[\cite{Jones80}] If $\Omega$ is a $C$-uniform domain, there is a bounded linear extension operator from $\BMO(\Omega)$ onto $\BMO(\bR^{d})$ whose norm depends quantitatively on $d$ and $C$.
\label{t:Jones}
\end{theorem}
We can actually write down the extension operator explicitly. First, let $W=W(\Omega)$ be the Whitney cube decomposition for the domain $\Omega$ (c.f. \cite[p. 16]{LittleStein}). The cubes $W$ have disjoint interiors, $\bigcup_{Q\in W}Q=\Omega$, and satisfy that for all $z\in Q\in W$, 
\begin{equation}\ell(Q)\sim_{d} \dist(Q,\d \Omega)\sim_{d} \dist (z,\d\Omega)
\label{e:whitney}
\end{equation}
with constants independent of $\Omega$ (such a collection exists for any open set $\Omega$).

Let $W'$ be a Whitney decomposition for $(\Omega^{c})^{\circ}$, and for each $Q'\in W'$, assign to it the closest cube $T(Q')\in W$ such that $\ell(T(Q'))\geq \ell(Q')$. If $\Omega$ is an {\it unbounded} $C$-uniform domain, then we may always find such a $T(Q')$. To see this, pick $x,y\in \Omega$ and connect then by a curve satisfying the properties in Definition \ref{d:uniform-domain}. Pick $z\in \gamma$ such that $\min\{|x-z|,|y-z|\}\geq \frac{|x-y|}{2}$. Then $z$ is contained in a Whitney cube $Q$ such that 
\[\dist(Q,\d\Omega)\sim_{d} \dist(z,\d\Omega)\gec |x-y|.\] 
Thus, if $\Omega$ is unbounded, we may pick $x$ and $y$ so that $|x-y|$ is as large as we like, and hence, for any $Q'\in W'$, we may always find $T(Q')\in W$ with $\ell(T(Q'))\geq \ell(Q')$.

Now we can write down the extension. For $f\in \BMO(\Omega)$, define (c.f. \cite[p. 54-57]{Jones80})
\begin{equation}
\tilde{f}= f+\sum_{Q'\in W'}\avint{T(Q')}f.
\label{e:ftwid}
\end{equation}
We claim that there is $C''$ depending only $C$ and $d$ so that
\begin{equation}
T(Q')\subseteq C''Q'\mbox{ and } Q'\subseteq C''T(Q') \mbox{ whenever } \ell(T(Q'))\geq\ell(Q').
\label{e:C''}
\end{equation}
By Lemma 2.10 in \cite{Jones80}, we know 
\begin{equation}
\mbox{ if }\ell(T(Q'))\geq \ell(Q'),\mbox{ then } d(Q',T(Q'))\lesssim_{C,d} \ell(Q')\leq \ell(T(Q')).
\end{equation} 
(which holds regardless of whether $\Omega$ is bounded or unbounded) and by virtue of $T(Q')$ and $Q'$ being Whitney cubes in $\Omega$ and $(\Omega^{c})^{\circ}$ respectively, we also know $\ell(T(Q'))\sim_{C,d}\ell(Q)$, which implies \eqn{C''}. 

Hence, if $x\in Q'\in W'$,
\begin{equation}
|\tilde{f}(x)|\leq \avint{T(Q')}|f|\lec_{C'',d} \avint{C''Q'}|f|\lec_{d} Mf(x),
\label{e:Mf}
\end{equation}
where, by an abuse of notation, $Mf$ now denotes the uncentered maximal function (that is, $Mf(x)$ is the supremum of averages over {\it all} cubes containing $x$, not just dyadic cubes). Moreover, we are considering $f$ to be defined over $\Real^d$ with $f =0$ in $\Omega^{c}$. 
Hence, by Theorem \ref{t:Jones}, \eqn{p<bmop0} (which we have proved for functions in $\BMO(\bR^{d})$), and the Hardy-Littlewood maximal inequality,
\begin{align*} 
||f||_{L^p(\Omega)} \leq ||\tilde{f}||_{L^p(\Real^d)} &
\stackrel{\eqn{p<bmop0}}{\lec}_{p_{0},d} p^{1-\frac{p_{0}}{p}}||\tilde{f}||_{\BMO(\bR^{d})}^{1-\frac{p_{0}}{p}}||\tilde{f}||_{L^{p_{0}}(\Real^d)}^{\frac{p_{0}}{p}} \\ 
& \stackrel{\eqn{Mf}}{\lec}_{C,d} p^{1-\frac{p_{0}}{p}}||f||_{\BMO(\Omega)}^{1-\frac{p_{0}}{p}}||Mf||_{L^{p_{0}}(\Real^d)}^{\frac{p_{0}}{p}}
\lec_{d,p_{0}}||f||_{\BMO(\Omega)}^{1-\frac{p_{0}}{p}}||f||_{L^{p_{0}}(\Real^d)}^{\frac{p_{0}}{p}}.
\end{align*}

We now prove Lemma \ref{lem:BMO_linear_p} in the case of a {\it bounded} $C$-uniform domain. In this setting (again, citing the construction in \cite{Jones80}), the extension $\tilde{f}$ is defined in a similar way to before, but we need to address the fact that for any cube $Q'$ there might not be a cube $T(Q')\in \Omega$ with $\ell(T(Q'))\geq \ell(Q')$ (recall that, before, the unboundedness of $\Omega$ played a role in showing that one existed). 

First, we claim that there is $C'$ depending only on $C$ and $d$ so that any bounded $C$-uniform domain $\Omega$ contains a Whitney cube $Q_{0}$ so that 
\begin{equation} Q_{0} \subseteq \Omega \subseteq C'Q_{0}\label{e:C'}\end{equation}
To see this, pick $x,y\in\Omega$ so that $|x-y|\geq \frac{\ell(\Omega)}{2}$. Then, just as we did earlier, we can find $Q_{0}\subseteq \Omega$ such that 
\[\ell(Q_{0})\gec_{C,d}|x-y|\gec \ell(\Omega),\] 
which proves the claim. 

Next, the proof is simpler if we scale $\bR^{d}$ by a factor $t=\ps{\frac{|\Omega|^{\frac{1}{p}}}{|Q_{0}|^{\frac{1}{p_{0}}}}}^{d\ps{\frac{1}{p_{0}}-\frac{1}{p}}}$ so that $|t\Omega|^{\frac{1}{p}}=|tQ_{0}|^{\frac{1}{p_{0}}}.$ 
Notice \eqn{p<bmop0} is invariant under such dilations, and $t\Omega$ remains $C$-uniform with the same constant $C$. It is easy to verify that the collection $tW=\{tQ:Q\in W\}$ still satisfies the properties of a Whitney cube decomposition for $t\Omega$ (with constants depending only on $d$ and not on $t$). To simplify notation, we will suppress the value $t$ below, or assume without loss of generality that $t=1$, and henceforth assume that 
\begin{equation}
|\Omega|^{\frac{1}{p}}=|Q_{0}|^{\frac{1}{p_{0}}}.
\label{e:weird-dilation}
\end{equation}

Now for each $Q'\in W'$, assign to it a cube $T(Q')\in W$ closest to $Q'$ such that $\ell(T(Q'))\geq \ell(Q')$ if such a cube exists; if not, choose $T(Q')=Q_{0}$.

Let $f\in \BMO(\Omega)$, and define $\tilde{f}$ exactly as in \eqn{ftwid} (which the association $Q'\rightarrow T(Q')$ just defined). Applying Theorem \ref{t:Jones} directly to $f$ would again immediately imply \eqn{p<bmop0} if we knew $||\tilde{f}||_{p_{0}}\lesssim ||f||_{p_{0}}$, but $\tilde{f}$ is constant and potentially nonzero outside far away from $\Omega$. To remedy this, we instead consider $g=f-f_{Q_{0}}$ extended to be zero in $\Omega^c$ and our first goal is to show that for all $x\in\bR^{d}$,
\begin{equation}
\abs{\tilde{g}(x)} \lec_{C,d} Mg(x).
\label{e:Mg}
\end{equation}
Establishing this estimate is similar to \eqref{e:Mf}.
Let $x\in Q'\in W'$. If $T(Q')<\ell(Q')$, then $T(Q')=Q_{0}$ and \eqn{Mg} is trivial since $\tilde{g}(x)=0$. Otherwise,  $T(Q')\geq \ell(Q')$ and we may apply \eqn{C''} and proceed as we did in \eqn{Mf} to obtain
\[|\tilde g(x)|\leq \avint{T(Q')}|g|\lesssim_{C'',d} \avint{B(x,C''\ell(Q'))}|g|\leq Mg(x),\]
which proves \eqn{Mg}.

Now we may proceed:
\begin{align}
||\tilde{g}||_{L^{p_{0}}(\Real^d)} 
\stackrel{\eqn{Mg}}{\lesssim} ||Mg||_{L^{p_{0}}(\Real^d)}
\lesssim_{p_{0},d}||g||_{L^{p_{0}}(\Omega)}
& \leq ||f||_{L^{p_{0}}(\Omega)}+|\Omega|^{\frac{1}{p_{0}}}\avint{Q_{0}}|f|\\ & 
\leq ||f||_{L^{p_{0}}(\Omega)} + |\Omega|^{\frac{1}{p_{0}}}\ps{\avint{Q_{0}}|f|^{p_{0}}}^{\frac{1}{p_{0}}}\\ & 
\leq  ||f||_{L^{p_{0}}(\Omega)} + \ps{\frac{|\Omega|}{|Q_{0}|}}^{\frac{1}{p_{0}}}\ps{\int_{Q_{0}}|f|^{p_{0}}}^{\frac{1}{p_{0}}}\\ & 
\stackrel{\eqn{C'}}{\lesssim}_{C',d} ||f||_{L^{p_{0}}(\Omega)}.
\label{e:gtwid-p0}
\end{align}
Thus,
\begin{align*}
||f||_{L^p(\Omega)}& 
\leq ||g||_{L^p(\Omega)}+|\Omega|^{\frac{1}{p}} \avint{Q_{0}}|f|\\ & 
\leq ||\tilde{g}||_{L^p(\Real^d)} +|\Omega|^{\frac{1}{p}}\ps{\avint{Q_{0}}|f|^{p_{0}}}^{\frac{1}{p_{0}}}\\ & 
\stackrel{\eqn{weird-dilation},\eqref{e:bmo+p0}}{\lesssim_{d,p_0}} p||\tilde{g}||_{\BMO(\bR^{d})}|| + ||\tilde{g}||_{L^{p_{0}}(\Real^d)} + ||f||_{L^{p_{0}}(\Omega)}\\ & 
\stackrel{\eqn{gtwid-p0}}{\lesssim}_{C',d} p||g||_{\BMO(\Omega)}+  ||f||_{L^{p_{0}}(\Omega)}\\ & 
=p||f||_{\BMO(\Omega)}+||f||_{L^{p_{0}}(\Omega)}
\end{align*}
where in the last line we used the fact that the BMO norm of a function is unaffected by adding or subtracting a constant (in this case, the constant is $f_{Q_{0}}$). The penultimate line used Theorem \ref{t:Jones}. 
We have now arrived at \eqn{bmo+p0}, and from here we may proceed as before in the $\Real^d$ case. 
\end{proof}

The next lemma provides straightforward sufficient conditions to be in the Hardy space $\mathcal{H}^1$, which arises when treating the $\Real^2$ case in Theorem \ref{thm:Hdot1Inviscid} below. It is not relevant for the treatment of bounded domains. 
The lemma is also general enough to find use elsewhere \cite{BR11}. 

\begin{lemma} \label{lem:Hardy}
Let $f \in L^1 \cap L^p$ for some $p > 1$ and satisfy $\int f dx = 0$, $\mathcal{M}_1 = \int \abs{x}\abs{f(x)}dx < \infty$. 
Then $f \in \mathcal{H}^1$ and 
\begin{equation*}
\norm{f}_{\mathcal{H}^1} \lesssim_{d,p} \norm{f}_{p} + \mathcal{M}_1. 
\end{equation*}
\end{lemma}
\begin{proof}
We will prove this using the duality between BMO and $\mathcal{H}^{1}$. 
Therefore, let $\norm{\K}_{BMO} = 1$. 
In what follows, we denote $\K_R := \frac{1}{\abs{B_R(0)}}\int_{\abs{x}\leq R} \K(x) dx$. 
By the mean-zero condition on $f$,
\begin{align*}
\abs{\int \K f dx} & \leq \int_{\abs{x} \leq 1}\abs{\K - \K_1}\abs{f}dx + \int_{\abs{x} > 1} \abs{\K - \K_1}\abs{f} dx. \\
& := T_{1} + T_{2}. 
\end{align*}
By H\"older and $p > 1$ with $\norm{\K}_{BMO} = 1$, we have  
\begin{align*}
T_{1} & \leq \norm{f}_p \left(\int_{\abs{x} \leq 1}\abs{\K - \K_1}^{p^\prime} dx\right)^{1/p^\prime} \\
& \lesssim_{p^\prime} \norm{f}_p. 
\end{align*}
The second inequality can be found in the proof of the John-Nirenberg inequality in \cite{BigStein}. 
We now deal with the second term.
Define the dyadic annuli $A_n := \set{x \in \Real^d: 2^n < \abs{x} < 2^{n+1}}$ for $n \geq 0$. 
Let $E_n := \set{x \in A_n: \abs{\K - \K_1} > 2^n}$. By definition we have, then 
\begin{align*}
T_{2} = \sum_{n \geq 0} \int_{E_n}\abs{\K - \K_1}\abs{f} dx + \int_{A_n \setminus E_n}\abs{\K - \K_1}\abs{f} dx = T_{21}+T_{22}
\end{align*}
The second term can be estimated via, 
\begin{equation}
T_{22} \leq 2^n \int_{A_n} \abs{f} dx 
  \leq \int_{A_n} \abs{x}\abs{f(x)} dx. 
  \label{ineq:T_{2}p1}
\end{equation}
Since $\K \in BMO$, we have $\abs{\K_{2^{n+1}} - \K_{2^n}} \lesssim_{d} 1$, and therefore $\abs{\K_{2^n} - \K_1} \lesssim_{d} n$.  
Applying this to $T_{21}$ implies, 
\begin{align*}
T_{21} & \leq \left(\int_{E_n} \abs{\K - \K_1}^{p^\prime} dx \right)^{1/p^{\prime}}\norm{f}_p \\
& =  \left(p^\prime \int_{2^n}^\infty \abs{ \set{\abs{\K - \K_1} > \lambda} \cap A_n } \lambda^{p^\prime - 1} d\lambda \right)^{1/p^\prime}\norm{f}_p \\
& \leq \left(p^\prime \int_{2^n}^\infty \abs{ \set{\abs{\K - \K_{2^n}} > \lambda - Cn} \cap A_n } \lambda^{p^\prime - 1} d\lambda \right)^{1/p^\prime}\norm{f}_p, 
\end{align*}
for some $C=C(d)> 0$. Using the John-Nirenberg inequality \cite{BigStein}, there is some $c=c(d)>0$ such that
\begin{align*}
\int_{E_n}\abs{\K - \K_1}\abs{f} dx & \lesssim \left(\abs{A_n}\int_{2^n}^\infty e^{-c(\lambda- Cn)} \lambda^{p^\prime -1} d\lambda \right)^{1/p^\prime}\norm{f}_p \\
& \lesssim \left(2^{nd}e^{Ccn} \int_{2^n}^\infty e^{-\lambda c}\lambda^{p^\prime - 1} d\lambda \right)^{1/p^\prime}\norm{f}_p.   
\end{align*}
Therefore, for large $n$ we have, 
\begin{equation}
\int_{E_n}\abs{\K - \K_1}\abs{f} dx \lesssim 2^{(d+p^\prime - 1)n/p^\prime}e^{Ccn/p^\prime} e^{-c2^n/p^\prime}\norm{f}_p. \label{ineq:T_{2}p2}   
\end{equation}
Summing \eqref{ineq:T_{2}p1} and \eqref{ineq:T_{2}p2} over $n$ then implies
\begin{equation*}
  T_{2} \lesssim_p \norm{f}_p + \mathcal{M}_1. 
\end{equation*}
\end{proof} 


\section{Uniqueness of weak solutions to \eqref{def:Type1} and \eqref{def:Type2}} \label{sec:Type1and2}
We will require that the nonlocal linear operator $V$ satisfy the following three conditions. 
\begin{itemize}
\item[\textbf{C1}] $\norm{V\rho}_q \lesssim \norm{\rho}_p$ with $1 + \frac{1}{q} = \frac{d-1}{d} + \frac{1}{p}$ and $1 < p < q < \infty$. 
\item[\textbf{C2}] The regularity requirement that $\grad V:BMO(\Real^d) \rightarrow BMO(\Real^d)$ is a bounded linear operator which additionally satisfies $\norm{\grad Vf}_p \lesssim p\norm{f}_p$ for all $p$ sufficiently large. 
\item[\textbf{C3}] The $\dot{H}^{-1}$ stability estimate $\norm{Vf}_{2} \lesssim \norm{f}_{\dot{H}^{-1}}$.  
\end{itemize}

\begin{remark}
The first condition can be weakened to require that $V:L^{q_1} \rightarrow L^{q_2}$ for any $1 < q_1,q_2 < \infty$.
\end{remark}

\begin{remark} 
Condition \textbf{C1} is a basic integrability requirement which is mainly to control the decay of the velocity field at infinity. 
The Condition \textbf{C2} is what provides sufficient regularity on the vector field to ensure log-Lipschitz continuity via Lemma \ref{lem:BMO_logLip}.  
Only these two conditions are necessary to apply the proof of Theorem \ref{thm:Hdot1Inviscid} to linear problems in which the vector field is fixed independent of $\rho$. 
Condition \textbf{C3} will be used as a stability condition which controls the dependence of the vector field $V(\rho_1 - \rho_2)$ on $\rho_1-\rho_2$, and hence is necessary to deal with the nonlinear aspect of the problem.   
Condition \textbf{C3} holds in several classes in known cases, including:  
\begin{itemize}
\item[(a)] If $V$ is a Fourier multiplier then \textbf{C2} implies \textbf{C3}, since differentiation and $V$ will commute. This covers all cases of the form $V\rho = \vec{K} \ast \rho$ for any $\vec{K}$ such that $\grad K$ is a singular integral operator of Calder\'on-Zygmund type. 
Common examples include: the 2D Euler equations, a number of parabolic-elliptic Patlak-Keller-Segel systems and many general aggregation equations. 
\item[(b)] $V\rho = \grad c$ and $Lc = f$ where $L$ is a second-order uniformly elliptic PDE will satisfy \textbf{C3} with minimal conditions on the coefficients. For example for $Lc = -\grad \cdot (A(x)\grad c) + \gamma(x)c = f$, it suffices to take $A,\gamma \in L^\infty$ with $A$ symmetric and uniformly positive definite and $\gamma(x) \geq 0$. Hence in this case it seems \textbf{C2} is more stringent than \textbf{C3}.  
Theorem \ref{thm:Hdot1Inviscid} is applied in this context in \cite{BR11}. 
\item[(c)] If $V\rho$ is a pseudo-differential operator in the standard symbol class $S^{-1}$ (as defined in \cite{BigStein}) then it will satisfy both \textbf{C2} and \textbf{C3}.
\end{itemize} 
Of course, \textbf{C3} fails for many models. For example, in the case of the 2D SQG equations, $V\rho = \grad^\perp (-\Delta)^{-1/2}\rho = (-R_2\rho, R_1\rho)$, and \textbf{C3} fails in this case (it holds with $\dot{H}^{-1}$ replaced by $L^2$, which allows Theorem \ref{thm:L2Inviscid}). It can also fail in the case of (b) above if $A(x)$ is only symmetric positive semi-definite. 
However, we should note that we are unaware of an example in which \textbf{C2} holds and \textbf{C3} fails, unless one is interested in extending this method to nonlinear nonlocal operators $V[\rho]$, where \textbf{C3} must be replaced by $\norm{V[\rho_1] - V[\rho_2]}_2 \lesssim \norm{\rho_1 - \rho_2}_{\dot{H}^{-1}}$.
\end{remark} 

\begin{theorem}[$\dot{H}^{-1}$ Method] \label{thm:Hdot1Inviscid} 
Suppose $V$ satisfies properties \textbf{C1-3}   
and consider the active scalar, 
\begin{equation}
\rho_t + \grad \cdot (\rho V\rho) = \Delta A(\rho), \label{def:ActScalar1}
\end{equation} 
where $A:\Real \rightarrow [0,\infty)$ is measurable and non-decreasing (possibly zero). 
Suppose there exists two weak solutions $\rho_1(t),\rho_2(t)$ defined on $[0,T]$ which satisfy $\rho_i(t) \in L^{2}_t(0,T;BMO(\Real^d)) \cap C_t([0,T];L^{p_0}_x(\Real^d))$ where if $d \geq 3$ we assume $1 \leq p_0 < 2d/(d+2)$ and if $d =2$ we take $p_0 = 1$. In $d = 2$ we additionally assume $\rho_i(x)x \in L^\infty_t(0,T;L^1_x(\Real^d))$. 
Then if $\rho_1(0) = \rho_2(0)$ then $\rho_1(t) = \rho_2(t)$ on $[0,T]$. 
\end{theorem}
\begin{proof}
Define $w := \rho_1 - \rho_2$ and let $\phi = -\mathcal{N}\ast w$.
First, notice by Lemma \ref{lem:BMO_linear_p} that $\rho_i(t) \in L^{2}_t(0,T;L^{p}_x)$ for all $p > p_0$.
If $d \geq 3$, note also that $p_0/(p_0 - 1) > 2d/(d-2)$. 
By Young's inequality we have, 
\begin{equation*}
\norm{\grad \phi(t)}_\infty \leq \norm{\grad\mathcal{N}\mathbf{1}_{B_1(0)}}_{\frac{2d}{2d-1}}\norm{w}_{2d} + \norm{\grad \mathcal{N}\mathbf{1}_{\Real^d \setminus B_1(0)}}_{\frac{p_0}{p_0 - 1}}\norm{w}_{p_0}, 
\end{equation*}
which implies $\grad \phi \in L^2_tL^\infty_x$. 
If $d \geq 3$ we similarly have, 
\begin{equation*}
\norm{\phi(t)}_\infty \leq \norm{\mathcal{N}\mathbf{1}_{B_1(0)}}_{\frac{2d}{2d-1}}\norm{w}_{2d} + \norm{\grad \mathcal{N}\mathbf{1}_{\Real^d \setminus B_1(0)}}_{\frac{p_0}{p_0 - 1}}\norm{w}_{p_0}, 
\end{equation*}
and therefore $\phi(t) \in L^{2}_tL^\infty_x$. 
In $\Real^2$, by Lemma \ref{lem:Hardy}, $w(t) \in \mathcal{H}^1$ uniformly for a.e. time, hence, $\phi(t) \in L^\infty_tL^\infty_x$ from the duality of $\mathcal{H}^1$ and $BMO$, 
\begin{equation*}
\abs{\phi(t,x)} = \abs{\int \frac{1}{2\pi} \log \abs{x-y} w(y) dy} \lesssim \norm{w}_{\mathcal{H}^1}.  
\end{equation*}
In particular, $-\Delta \phi$ defines a bounded distribution and hence we have $-\Delta \phi = w$ in the weak sense.  
In $d \geq 3$ note by Young's inequality for $L^{\frac{d}{d-1},\infty}$, 
\begin{align*}
\norm{\grad \phi(t)}_{2} \lesssim \norm{w}_{\frac{2d}{d+2}}\norm{\mathcal{\grad N}}_{\frac{d}{d-1},\infty}, 
\end{align*}
hence $\grad \phi \in L_t^2L_x^2$. 
Proving this in $d = 2$ is more delicate as it does not follow from only $L^p$ estimates on $w$; we follow a procedure similar to that taken in \cite{BertozziBrandman10}. 
By Parseval's theorem it suffices to prove that $\widehat{\grad \phi} \in L^2_tL_x^2$. First, notice 
\begin{align*}
\int \abs{\grad \phi (t,x)}^2 dx = \int \abs{\widehat{\grad \phi}(t,\xi)}^2 d\xi \lesssim \int_{\abs{\xi} \leq 1} \frac{\abs{\hat{w}(t,\xi)}^2}{\abs{\xi}^2} d\xi + \int \abs{\hat{w}(t,\xi)}^2 d\xi. 
\end{align*} 
The latter term is controlled by Parseval's theorem since $w \in L^2_tL_x^2$. The first term is not a priori controlled as $\abs{\xi}^{-2}$ is not integrable in $\Real^2$, however, it suffices to prove that $\hat{w}(t)$ is uniformly Lipschitz continuous.  
By the Riemann-Lebesgue lemma, $w(t,x) \in C_t([0,T];L^1_x(\Real^2))$ implies $\hat{w}(t,\xi) \in C_t([0,T];C_x(\Real^2))$. 
Since \eqref{def:ActScalar1} is in divergence form, we have $\int w(t,x) dx \equiv 0$ and therefore $\hat{w}(t,0) \equiv 0$.
The Lipschitz continuity follows from the control on the first moment: 
let $\xi_1,\xi_2 \in \Real^2$, then by the mean value theorem for $i \in \set{1,2}$, 
\begin{align*}
\abs{\hat{\rho}_i(\xi_1) - \hat{\rho}_i(\xi_2)} & \leq \frac{1}{(2\pi)^{d/2}}\abs{\int \rho_i(x)\left(e^{-ix\cdot\xi_1} - e^{-ix\cdot\xi_2}\right) dx} \\
& \leq \frac{1}{(2\pi)^{d/2}}\int \abs{\rho_i(x)}\abs{1 - e^{ix\cdot(\xi_1 - \xi_2)}}dx \\
& \lesssim \abs{\xi_1 - \xi_2}\int \abs{\rho_i(x)} \abs{x} dx\\ & = \mathcal{M}_1(\rho_i)\abs{\xi_1 - \xi_2}. 
\end{align*}
Hence, $\grad \phi(t) \in L^2_x$ uniformly in time. 
Therefore, regardless of dimension we have that $\grad \phi \in L^2_tL_x^2 \cap L^2_tL_x^\infty$ and we can finally compute the evolution. 
By the regularity assumptions of the solution and the regularity properties of $\phi$ deduced above, we may use $\phi$ as a test function in the definition of weak solution and compute the time evolution of $\norm{\grad \phi(t)}_2$, 
\begin{align*}
\frac{1}{2}\frac{d}{dt}\int \abs{\grad \phi}^2 dx & = \int \grad \phi \cdot \grad \phi_t dx = \int \phi w_t dx \\
& = \int \phi \grad \cdot \left(\grad A(\rho_1) - \grad A(\rho_2) - w V\rho_1 - \rho_2(V\rho_1 - V\rho_2) \right) dx \\
& = -\int \left(A(\rho_1) - A(\rho_2)\right)(\rho_1 - \rho_2) dx + \int \grad \phi \cdot wV\rho_1 dx + \int \grad \phi \cdot \rho_2(V\rho_1 - V\rho_2) dx \\ & := T1 + T2 + T3.   
\end{align*}
Since $A$ is non-decreasing with have, 
\begin{equation*}
T1 = - \int \left(A(\rho_1) - A(\rho_2)\right)(\rho_1 - \rho_2) dx \leq 0. 
\end{equation*}
We now deal with $T2$, which is most related to the regularity of the advective velocity field. Since $w = -\Delta \phi$ and the boundary terms vanish due to the integrability conditions on $w$, 
\begin{align*}
T2 & = -\sum_{ij} \int \partial_{ii} \phi \partial_j \phi (V\rho_1)_j dx \\ 
& = \sum_{ij} \int \partial_i \phi \partial_{ij} \phi (V\rho_1)_j + \partial_i\phi \partial_j\phi \partial_i (V\rho_1)_j dx \\
& = \sum_{ij} \int \partial_j\left(\frac{1}{2} \abs{\partial_i\phi}^2 \right) (V\rho_1)_j + \partial_i\phi \partial_j\phi \partial_i (V\rho_1)_j dx \\
& \lesssim \int \abs{\grad V\rho_1}\abs{\grad \phi}^2 dx. 
\end{align*}
Using $\grad V\rho_1 \in L^2_tBMO_x$ and Lemma \ref{lem:BMO_linear_p} we then have, 
\begin{align*}
\int \abs{\grad V\rho_1}\abs{\grad \phi}^2 dx & \leq \norm{\grad V\rho_1}_p \left( \int \abs{\grad \phi}^{2p/(p-1)} dx \right)^{1-1/p} \\
& \lesssim \left( p\norm{\grad V\rho_1}_{BMO} + \norm{\grad V\rho_1}_{p_0} \right)\left(\norm{\grad \phi}_\infty^{2p/(p-1)- 2}\int \abs{\grad \phi}^{2} dx \right)^{1-1/p} \\
& \lesssim \left( p\norm{\grad V\rho_1}_{BMO} + 1 \right)\norm{\grad \phi}_\infty^{1/p}\norm{\grad \phi}_2^{2-2/p}.
\end{align*} 
Now we turn to $T3$, which regards the stability of the map $\rho \mapsto V\rho$. By Lemma \ref{lem:BMO_linear_p} again we have, 
\begin{align*}
\abs{T3} & \leq \norm{\rho_2}_p\norm{\grad \phi \cdot (V\rho_1 - V\rho_2)}_{p/(p-1)} \\
& \lesssim \left( p\norm{\rho_2}_{BMO} + \norm{\rho_2}_{p_0}\right)\norm{\grad \phi}_{2p/(p-1)}\norm{V\rho_1 - V\rho_2}_{2p/(p-1)} \\ 
& \lesssim \left( p\norm{\rho_2}_{BMO} + 1\right)\norm{\grad \phi}_{\infty}^{1/p}\norm{V\rho_1 - V\rho_2}_{\infty}^{1/p}\norm{\grad \phi}_2^{1-1/p}\norm{V\rho_1 - V\rho_2}_{2}^{1-1/p}. 
\end{align*} 
Assumptions \textbf{C1} and \textbf{C2} together with Morrey's inequality imply $V\rho_i \in L^2_tL_x^\infty$. 
Assumption \textbf{C3} implies $\norm{V(\rho_1 - \rho_2)}_2 \lesssim \norm{\grad \phi}_2$. 
Therefore, for all $p$ sufficiently large we have
\begin{align*}
\frac{d}{dt}\norm{\grad \phi(t)}_2^2 \leq p f(t) \norm{\grad \phi(t)}_2^{2-2/p}, 
\end{align*} 
for some positive $f(t) \in L^{2/(1+2/p)}_{t}$. Integrating, 
\begin{align*}
\norm{\grad \phi(t)}_2^{2/p} - \norm{\grad \phi(0)}_2^{2/p} \leq \int_0^t f(s) ds. 
\end{align*}
Since $\rho_1(0) = \rho_2(0)$ this implies, 
\begin{align*}
\norm{\grad \phi(t)}_2^2 \leq \left( \int_0^t f(s) ds \right)^{p} \leq \left(\norm{f}_{L^{3/2}_t(0,T)} t^{1/3}\right)^{p}.
\end{align*}
Hence, if we restrict $t \leq t_0 <  2^{-3}\norm{f}^{-3}_{L^{3/2}_t(0,T)}$, we have $\norm{\grad \phi(t)}_2^2 \leq 2^{-p}$ we may send $p \rightarrow \infty$ and prove $\norm{\grad \phi(t)}_2 \equiv 0$ a.e. in time $0 \in [0,t_0]$, however clearly we can simply iterate and prove that $\norm{\grad \phi(t)}_2 \equiv 0$ on $[0,T]$ and therefore $w (t) \equiv 0$.
\end{proof} 

The following result combines the methods of the above theorem with an $L^2$ estimate to provide a simplified and more general proof of the results in \cite{Rusin11}. 
In particular, we do not need to assume that logarithmic singularities in the derivatives occur on a bounded set.

\begin{theorem}[$L^2$ Estimates] \label{thm:L2Inviscid}
Let $V$ be a bounded linear operator $V:BMO(\Real^d;\Real) \rightarrow BMO(\Real^d;\Real^d)$ which satisfies $\norm{Vf}_p \lesssim p\norm{f}_p$ for all $2 \leq p < \infty$ and $\grad \cdot Vf \equiv 0$. 
Consider the active scalar, 
\begin{equation*}
\rho_t + V\rho\cdot \grad \rho = -\nu(-\Delta)^{\gamma}\rho,  
\end{equation*} 
for some $\gamma \in (0,1]$ and $\nu \geq 0$. 
Suppose there exists two weak solutions $\rho_1(t),\rho_2(t)$ defined on $[0,T]$ which satisfy $\rho_i(t) \in C_t([0,T];L^{2}_x(\Real^d))$ and $\grad \rho_i \in L^2_t(0,T;BMO(\Real^d)\cap L^{p_0}(\Real^d))$ for some $p_0 \in [1,\infty)$. Then if $\rho_1(0) = \rho_2(0)$ then $\rho_1(t) = \rho_2(t)$ on $[0,T]$.
\end{theorem}
\begin{remark}
It will be evident from the proof that we need only assume $\grad \rho \in L^{1+}_t(0,T;BMO \cap L^{p_0})$.  
\end{remark}
\begin{remark}
If $\nu > 0$ then we may potentially relax the conditions of this theorem; see \S\ref{sec:Dissipation} below. However, 
note that depending on the decay of $\grad \rho$, there are still dissipative cases covered by this theorem that are not covered by Theorem \ref{thm:L2Dissip} below.     
\end{remark}
\begin{proof}
We estimate the $L^2$ norm of $w := \rho_1 - \rho_2$. 
First, notice that by Lemma \ref{lem:BMO_logLip}, $\rho_i, w \in L^2_tL_x^\infty$ and are in fact log-Lipschitz continuous for a.e. time.  
Using the divergence free condition, 
\begin{align*}
\frac{1}{2}\frac{d}{dt}\int \abs{w}^2 dx & = -\nu\int w(-\Delta)^{\gamma}w dx - \int w V\rho_2\grad w + w\grad \rho_2 \cdot Vw dx\\
& = -\nu\int \abs{\abs{\grad}^{\gamma}w}^2 dx + \int w \grad \rho_2 \cdot Vw dx \\
& \leq \int w \grad \rho_2 \cdot Vw dx. 
\end{align*}
Using Lemma \ref{lem:BMO_linear_p} for $p > p_0$ and the boundedness of the mapping $V:L^{2p/(p-1)} \rightarrow L^{2p/(p-1)}$,  
\begin{align*}
\frac{1}{2}\frac{d}{dt}\int \abs{w}^2 dx & \leq \norm{\grad \rho_2}_p \left[ \int \abs{w Vw}^{\frac{p}{p-1}}dx \right]^{1 - 1/p} \\
& \lesssim \left(p\norm{\grad \rho_2}_{BMO} + \norm{\grad \rho_2}_{p_0}\right) \left[\int \abs{Vw}^{\frac{2p}{p-1}}dx \right]^{(p-1)/2p} \left[\int \abs{w}^{\frac{2p}{p-1}}dx \right]^{(p-1)/2p} \\
& \lesssim \left(p\norm{\grad \rho_2}_{BMO} + \norm{\grad \rho_2}_{p_0} \right)\left(\frac{2p}{p-1}\right)\left[\int \abs{w}^{\frac{2p}{p-1}}dx \right]^{(p-1)/p} \\ 
& \lesssim \left(p\norm{\grad \rho_2}_{BMO} + \norm{\grad \rho_2}_{p_0} \right)\norm{w}_\infty^{2/p} \norm{w}_2^{2-2/p}. 
\end{align*}
Since $\grad\rho_i \in L^2_t(0,T;BMO_x \cap L_x^{p_0})$, for $p$ sufficiently large we have, 
\begin{equation*}
\frac{1}{2}\frac{d}{dt}\norm{w}_2^2 \lesssim pf(t)\norm{w}_2^{2-2/p}, 
\end{equation*}
for some $0 < f(t) \in L^{2/(1+2/p)}_{t}(0,T)$. As above in the proof of Theorem \ref{thm:Hdot1Inviscid}, this suffices. 
\end{proof}

\section{Stability Results for Active Scalars with Dissipation} \label{sec:Dissipation} 
In \S\ref{sec:Type1and2}, the sign-definite dissipative effects were ignored and viscous and inviscid systems were treated in the same fashion. 
It is well-known that for the dissipative 2D quasi-geostrophic equations one can still have uniqueness at lower regularities than what is provided by Theorem \ref{thm:L2Inviscid} (see for example \cite{Wu05,Kiselev10} and the references therein).
In light of existing results on 2D SQG, the following easy general theorem is proved by a standard application of Gagliardo-Nirenberg or Sobolev embedding and essentially mirrors contraction mapping arguments for local existence and small data critical theory, but we state and prove it for completeness.    
\begin{theorem}[$L^2$ Estimates for Dissipative Active Scalars] \label{thm:L2Dissip}
Suppose there is some $q \geq d/(2\gamma)$ with $q > 1$ such that $V:L^{\frac{2q}{q-1}} \rightarrow L^{\frac{2q}{q-1}}$ is a bounded linear operator. 
Consider the dissipative active scalar, 
\begin{equation*}
\rho_t + V\rho\cdot \grad \rho = -\nu(-\Delta)^{\gamma}\rho,  
\end{equation*} 
for some $\gamma \in (0,1]$ and $\nu > 0$.
Suppose there exists two weak solutions $\rho_1(t),\rho_2(t)$ defined on $[0,T]$ which satisfy $\rho_i(t) \in C_t([0,T];L^{2}_x(\Real^d))\cap L_t^2(0,T;H^{\gamma})$ and $\grad \rho_i \in L^r_t(0,T;L_x^q)$ with $r = \frac{2\gamma q}{2\gamma q - d}$.   
Then, if $q > d/(2\gamma)$ we have the stability estimate, 
\begin{equation}
\norm{\rho_1(t) - \rho_2(t)}_2 \leq \exp \left[ \nu^{-\frac{d}{2\gamma q - d}} C(d,q,\norm{\grad \rho_i}_{L_t^rL_x^q},\norm{V}_{L^{2q/(q-1)} \rightarrow L^{2q/(q-1)}})\right]\norm{\rho_1(0) - \rho_2(0)}_2. \label{ineq:DissipStab_L2}
\end{equation} 
If $d > 2\gamma$ and the critical norm satisfies $\norm{\grad \rho_i}_{L^\infty_tL_x^{\frac{d}{2\gamma}}} \leq \nu \epsilon_0(d,\norm{V}_{L^{2d/(d-2\gamma)} \rightarrow L^{2d/(d-2\gamma)}})$, for some absolute $\epsilon_0$ depending only on those quantities indicated, then the solution map is a contraction in $L^2$, that is, for all $t > s \geq 0$,  
\begin{equation*}
\norm{\rho_1(t) - \rho_2(t)} \leq \norm{\rho_1(s) - \rho_2(s)}_{2}, \;\; \forall t \geq s. \label{ineq:DissipContract_L2}
\end{equation*}
\end{theorem} 
\begin{proof} 
Beginning as before we define $w := \rho_1 - \rho_2$ and using the divergence free condition, 
\begin{align*}
\frac{1}{2}\frac{d}{dt}\int \abs{w}^2 dx & = -\nu\int w(-\Delta)^{\gamma}w dx - \int w V\rho_2\grad w + w\grad \rho_2 \cdot Vw dx\\ 
& = -\nu\int \abs{ \abs{\grad}^\gamma w}^2 dx + \int w \grad \rho_2 \cdot Vw dx,  
\end{align*}
where $\abs{\grad}$ is the Fourier multiplier $\abs{\grad} f := (2\pi)^{-d/2}\int \abs{\xi}\hat{f}(\xi)e^{ix\cdot \xi}d\xi$. 
First consider the subcritical case $q > \frac{d}{2\gamma}$. 
By Gagliardo-Nirenberg we have, 
\begin{align*}
\norm{w}_{\frac{2q}{q-1}} \leq C(d,q,\gamma)\norm{w}_2^{1-\frac{d}{2q\gamma}}\norm{\abs{\grad}^\gamma w}^{\frac{d}{2q\gamma}}_2, 
\end{align*}
which implies, 
\begin{align*}
\frac{1}{2}\frac{d}{dt}\int \abs{w}^2 dx & \leq -\nu\norm{\abs{\grad}^\gamma w}_2^2 + \norm{w}_{\frac{2q}{q-1}}\norm{Vw}_{\frac{2q}{q-1}}\norm{\grad \rho_2}_q \\ 
& \leq -\nu\norm{\abs{\grad}^\gamma w}_2^2 + \norm{w}_{\frac{2q}{q-1}}^2 \norm{V}_{L^{\frac{2q}{q-1}} \rightarrow L^{\frac{2q}{q-1}}}\norm{\grad \rho_2}_q \\
& \leq -\nu\norm{\abs{\grad}^\gamma w}_2^2 + C(d,q,\gamma)\norm{w}_2^{2-\frac{d}{q\gamma}}\norm{\abs{\grad}^\gamma w}_2^{\frac{d}{q\gamma}}\norm{V}_{L^{\frac{2q}{q-1}} \rightarrow L^{\frac{2q}{q-1}}}\norm{\grad \rho_2}_q.     
\end{align*}
By weighted Young's inequality, 
\begin{align*}
\frac{1}{2}\frac{d}{dt}\int \abs{w}^2 dx & \leq -\frac{\nu}{2}\norm{\abs{\grad}^\gamma w}_2^2 + \left(\frac{\nu}{2}\right)^{-\frac{d}{2\gamma q - d}}C\left(d,q,\gamma,\norm{V}_{L^{\frac{2q}{q-1}} \rightarrow L^{\frac{2q}{q-1}}}\right)\norm{\grad \rho_2}_q^{\frac{2\gamma q}{2\gamma q - d}}\norm{w}_2^{2} \\
& \leq \left(\frac{\nu}{2}\right)^{-\frac{d}{2\gamma q - d}}C\left(d,q,\gamma,\norm{V}_{L^{\frac{2q}{q-1}} \rightarrow L^{\frac{2q}{q-1}}}\right)\norm{\grad \rho_2}_q^{\frac{2\gamma q}{2\gamma q - d}}\norm{w}_2^{2}.   
\end{align*}
which implies the assertion \eqref{ineq:DissipStab_L2} by Gr\"onwall's inequality. 

We now prove the critical regularity assertion \eqref{ineq:DissipContract_L2}. 
Beginning as before we have, 
\begin{align*}
\frac{1}{2}\frac{d}{dt}\int \abs{w}^2 dx & = -\nu\int w(-\Delta)^{\gamma}w dx - \int w V\rho_2\grad w + w\grad \rho_2 \cdot Vw dx\\ 
& = -\nu\int \abs{ \abs{\grad}^\gamma w}^2 dx + \int w \grad \rho_2 \cdot Vw dx \\
& \leq -\nu\norm{\abs{\grad}^\gamma w}_2^2 + \norm{w}_{\frac{2d}{d-2\gamma}}^2 \norm{V}_{L^{\frac{2d}{d-2\gamma}} \rightarrow L^{\frac{2d}{d-2\gamma}}}\norm{\grad \rho_2}_{\frac{d}{2\gamma}}.
\end{align*}
By (homogeneous) Sobolev embedding we have,
\begin{align*} 
\frac{1}{2}\frac{d}{dt}\int \abs{w}^2 dx & = -\nu\norm{\abs{\grad}^\gamma w}_2^2 + C(\gamma,d)\norm{V}_{L^{\frac{2d}{d-2\gamma}} \rightarrow L^{\frac{2d}{d-2\gamma}}}\norm{\grad \rho_2}_{\frac{d}{2\gamma}}\norm{\abs{\grad}^\gamma w}^2_2. 
\end{align*}
Hence, provided, 
\begin{equation*}
C(\gamma,d)\norm{V}_{L^{\frac{2d}{d-2\gamma}} \rightarrow L^{\frac{2d}{d-2\gamma}}}\norm{\grad \rho_2}_{\frac{d}{2\gamma}} \leq \nu, 
\end{equation*} 
the solution map is a contraction in $L^2$. 
\end{proof}    

Using the Sobolev inequalities to capitalize on the stabilizing effects of the diffusion was straightforward in the previous theorem. 
A similar (but slightly less obvious) approach can also be adapted for the $\dot{H}^{-1}$ method used above in Theorem \ref{thm:Hdot1Inviscid}.
The proof is limited to linear diffusion since, due to the lack of an appropriate embedding theorem, there does not seem to be an obvious way to extract a stabilizing effect from the nonlinear diffusions.   
One way this theorem is useful is that it shows solutions to \eqref{def:Type1} with ``finite energy dissipation,'' $\int_0^t\int\abs{\rho}^2 +  \abs{\grad \rho}^2 dx dt < \infty$ are unique (see Remark \ref{rmk:EnergyDissipation} below). 
However, this theorem is obviously not as powerful as the deep methods of \cite{GallagherGallay05}, which ultimately prove that measure-valued solutions to the 2D Navier-Stokes equations in vorticity form are unique.
N. Masmoudi and one of the authors recently extended these methods to include the 2D parabolic-elliptic Patlak-Keller-Segel system in \cite{MasmoudiBedrossian12}.  
A primary aspect of these methods is the approximation of the contribution of the atomic part of the initial data as forward self-similar solutions and a good understanding of the linearization around these solutions, an approach which could prove difficult or impossible in some cases covered by Theorem \ref{thm:DissipHdot1}. 

\begin{theorem}[$\dot{H}^{-1}$ Method for Dissipative Active Scalars] \label{thm:DissipHdot1}
Suppose, if $d\geq 3$ that there is a $q \geq d/2$ or if $d = 2$ that there is a $q > 1$ such that
$\grad V:L^q \rightarrow L^q$ and $V \grad \cdot : L^{\frac{2q}{q-1}} \rightarrow L^{\frac{2q}{q-1}}$ are both bounded linear operators.
Consider the viscous active scalar,
\begin{equation*}
\rho_t + \grad \cdot (\rho V\rho) = \nu \Delta \rho, \;\; \nu > 0. 
\end{equation*} 
Suppose there exists two weak solutions $\rho_1(t),\rho_2(t)$ defined on $[0,T]$ which satisfy 
\begin{equation*}
\rho_i(t) \in L^{r}_t(0,T;L^{q}(\Real^d)) \cap L^2(0,T;L^2) \cap C_t([0,T];L^{\frac{2d}{d+2}}_x(\Real^d))
\end{equation*}
with  $r = 2q/(2q-d)$. 
In $d = 2$ we further assume $u(x)x \in L^\infty_t(0,T;L^1_x(\Real^d))$. 
Then we have the stability estimate
\begin{equation}
\norm{\rho_1(t) - \rho_2(t)}_{\dot{H}^{-1}} \leq \exp\left[\nu^{-\frac{d}{2q-d}}C(d,q,V)\norm{\rho_i}^r_{L^r_tL_x^q(0,t)}\right)\norm{\rho_1(0) - \rho_2(0)}_{\dot{H}^{-1}}. \label{ineq:stabilitydotH}
\end{equation}
If $d > 2$ and the critical norm satisfies $\norm{\rho_i(t)}_{d/2} \leq \nu\epsilon_0(d,V)$ for some $\epsilon_0 > 0$ depending only on the quantities specified for a.e. $t \in (0,T)$, then the solution map is a contraction in $\dot{H}^{-1}$, that is, for all $t > s \geq 0$,
\begin{equation*}
\norm{\rho_1(t) - \rho_2(t)}_{\dot{H}^{-1}} \leq \norm{\rho_1(s) - \rho_2(s)}_{\dot{H}^{-1}}. 
\end{equation*} 
\end{theorem}
\begin{remark} \label{rmk:EnergyDissipation}
Many solutions one can construct will not only have $\rho_i \in L^2(0,T;L^2)$ but also have $\grad \rho_i \in L^2(0,T;L^2)$.
In this case, Sobolev embedding (in $d \geq 3$) or Gagliardo-Nirenberg (in $d = 2$) imply that $\rho_i \in L^r_t(0,T;L_x^q)$ for a range of admissible pairs $(r,q)$ 
and hence all such active scalars are unique with only the assumption of $\rho_i \in L^2(0,T;H^1) \cap C_t([0,T];L_x^{2d/(d+2)}(\Real^d))$ (with the appropriate modification in 2D).  
An interesting side note is that in 2D, Theorem \ref{thm:DissipHdot1} is only slightly stronger than Theorem \ref{thm:Hdot1Inviscid}, as $H^1(\Real^2) \hookrightarrow BMO(\Real^2) \cap L^2(\Real^2)$ \cite{KozonoWadade08}. 
\end{remark} 
\begin{proof}
We prove the assertion \eqref{ineq:stabilitydotH} first. 
Let $w := \rho_1 - \rho_2$ and $-\Delta \phi = w$, equivalent to $-\mathcal{N} \ast w = \phi$ since
as in the proof of Theorem \ref{thm:Hdot1Inviscid} we have $\phi \in L_t^rL_x^\infty$. 
Moreover, $\grad \phi \in L_t^2 L_x^2$ and computing the evolution of $\norm{\grad \phi(t)}_2$ similar to Theorem \ref{thm:Hdot1Inviscid},
\begin{align*}
\frac{1}{2}\frac{d}{dt}\int \abs{\grad \phi}^2 dx & \leq -\nu\int \abs{w}^2 dx + 2\int \abs{\grad V\rho_1}\abs{\grad \phi}^2 dx + \int \rho_2 \grad \phi \cdot Vw dx  \\
& \leq -\nu\norm{w}_2^2 + 2\norm{\grad V\rho_1}_{q}\norm{\grad \phi}_{\frac{2q}{q-1}}^2 + \norm{\rho_2}_{q}\norm{Vw}_{\frac{2q}{q-1}}\norm{\grad \phi}_{\frac{2q}{q-1}} \\
& \leq -\nu\norm{w}_2^2 + 2\norm{\grad V}_{L^{q} \rightarrow L^{q}}\norm{\rho_1}_{q}\norm{\grad \phi}_{\frac{2q}{q-1}}^2 + \norm{\rho_2}_q\norm{V \grad \cdot}_{L^{\frac{2q}{q-1}} \rightarrow L^{\frac{2q}{q-1}}}\norm{\grad \phi}_{\frac{2q}{q-1}}^2. 
\end{align*}
By the Gagliardo-Nirenberg inequality and the Calder\'on-Zygmund inequality, 
\begin{equation*}
\norm{\grad \phi}_{\frac{2q}{q-1}} \lesssim_d \norm{\grad \phi}_2^{1-\frac{d}{2q}}\norm{D^2\phi}_2^{\frac{d}{2q}} \lesssim \norm{\grad \phi}_2^{1-\frac{d}{2q}}\norm{\Delta \phi}_2^{\frac{d}{2q}} = \norm{\grad \phi}_2^{1-\frac{d}{2q}}\norm{w}_2^{\frac{d}{2q}}. 
\end{equation*}
Hence, 
\begin{align*}
\frac{1}{2}\frac{d}{dt}\int \abs{\grad \phi}^2 dx & \leq -\nu \norm{w}_2^2 + C(d,\grad V)\norm{\rho_1}_q\norm{\grad \phi}_2^{2-\frac{d}{q}}\norm{w}_2^{\frac{d}{q}} + C(d,V \grad \cdot)\norm{\rho_2}_q\norm{\grad \phi}_2^{2-\frac{d}{q}}\norm{w}_2^{\frac{d}{q}}. 
\end{align*}
By weighted Young's inequality,
\begin{align*}
\frac{1}{2}\frac{d}{dt}\int \abs{\grad \phi}^2 dx & \leq -\frac{\nu}{2} \norm{w}_2^2  + \nu^{-\frac{d}{2q - d}}C(d,q,\grad V,V\grad \cdot) \left(\norm{\rho_1}^{\frac{2q}{2q - d}}_q + \norm{\rho_2}_q^{\frac{2q}{2q - d}}\right)\norm{\grad \phi}_2^{2}. 
\end{align*}
Hence, by Gr\"onwall's inequality,
\begin{equation*}
\norm{\grad \phi(t)}_2^2 \leq \norm{\grad \phi(0)}_2^2 \exp\left[\nu^{-\frac{d}{2q-d}}C(d,q,\grad V,V\grad \cdot)\int_0^t \left(\norm{\rho_1(s)}^{r}_{q} + \norm{\rho_2(s)}^{r}_{q}\right) ds\right]. 
\end{equation*} 
This proves the first (subcritical) assertion. 

The second (critical) assertion follows similarly, with modifications to deal with the criticality. 
Beginning as before, 
\begin{align*}
\frac{1}{2}\frac{d}{dt}\int \abs{\grad \phi}^2 dx & \leq -\nu\int \abs{w}^2 dx + 2\int \abs{\grad V\rho_1}\abs{\grad \phi}^2 dx + \int \rho_2 \grad \phi \cdot Vw dx  \\ 
 & \leq -\nu\int \abs{w}^2 dx + 2\norm{\grad V}_{L^{d/2} \rightarrow L^{d/2}}\norm{\rho_1}_{d/2}\norm{\grad \phi}_{\frac{2d}{d-2}}^2 \\ & \;\;\; + \norm{\rho_2}_{d/2}\norm{V\grad \cdot}_{L^{2d/(d-2)} \rightarrow L^{2d/(d-2)}}\norm{\grad \phi}_{\frac{2d}{d-2}}^2.
\end{align*}
By Sobolev embedding and the Calder\'on-Zygmund inequality again, 
\begin{align*}
\frac{1}{2}\frac{d}{dt}\int \abs{\grad \phi}^2 dx & \leq -\nu\norm{w}_2^2 + C(d)\left(\norm{\grad V}_{L^{\frac{d}{2}} \rightarrow L^{\frac{d}{2}}}\norm{\rho_1}_{\frac{d}{2}} + \norm{V\grad \cdot}_{L^{\frac{2d}{(d-2)}} \rightarrow L^{\frac{2d}{(d-2)}}}\norm{\rho_2}_{\frac{d}{2}}\right)\norm{w}_2^2. 
\end{align*}
Hence, the result follows immediately. 
\end{proof}

\section{Extension to Bounded Domains} \label{sec:BoundedDomains}
In this section we discuss how to extend Theorem \ref{thm:Hdot1Inviscid} to bounded domains with Lipschitz boundaries. 
Recall that a bounded domain $\Omega$ is a {\it Lipschitz domain} if, for every $\xi \in \d \Omega$, there is a neighborhood $U$ of $\xi$ and a Lipschitz function $A:\bR^{n-1}\rightarrow \bR$ such that,
\[U\cap \d \Omega \subseteq  T\{ (x,A(x)): x\in \bR^{n-1}\}\]
where $T$ is some linear transformation formed by a rotation and translation. 
Less regular domain boundaries may be treatable under certain conditions, but we will not consider these cases here. 
For certain PDE, the method would also require that $\Omega$ is convex or that $\Omega$ is simply connected.  

Following \cite{BertozziSlepcev10,BRB10}, the most important consideration is to replace $\dot{H}^{-1}$ with the following natural analogue: define 
\begin{align*} 
-\Delta \phi & = f- \int_\Omega f dx \\ 
\grad \phi \cdot \hat n|_{\partial \Omega} & = 0, 
\end{align*}
where $\hat n$ denotes the outward unit normal, then the $H^{-1}(\Omega)$ norm is given by
\begin{align*} 
\norm{f}_{H^{-1}(\Omega)} := \norm{\grad \phi}_{L^2(\Omega)}. 
\end{align*}  
We will need now the following three conditions, which are similar to conditions \textbf{C1-3} in \S\ref{sec:Type1and2}. 
\begin{itemize} 
\item[\textbf{B1}] The no-penetration condition: $V\rho \cdot \hat n|_{\partial \Omega} \leq 0$ for all solutions $\rho$. 
\item[\textbf{B2}] The regularity requirement that $\grad V:BMO(\Omega) \rightarrow BMO(\Omega)$ is a bounded linear operator which additionally satisfies $\norm{\grad Vf}_p \lesssim p\norm{f}_p$ for all $p$ sufficiently large.
\item[\textbf{B3}] The $H^{-1}$ stability estimate $\norm{Vf}_{2} \lesssim \norm{f}_{H^{-1}(\Omega)}$.
\end{itemize}
\begin{remark} 
All three conditions hold, for example, in the case of the 2D Euler/Navier-Stokes equations if $\Omega$ is simply connected. 
They also hold for the parabolic-elliptic Patlak-Keller-Segel model with homogeneous Neumann or Dirichlet data (the latter requiring non-negative solutions) on the chemo-attractant. 
See also \cite{BertozziSlepcev10,BRB10} for other examples, which require convexity of $\Omega$ in order to satisfy \textbf{B1}.
\end{remark} 

\begin{theorem}[$H^{-1}$ Method on Lipschitz Domains] \label{thm:Hdot1InviscidBddDomain} 
Suppose $\Omega$ is a bounded, Lipschitz domain, $V$ satisfies properties \textbf{B1-3}   
and consider the active scalar with the no-flux boundary condition, 
\begin{align*}
\rho_t + \grad \cdot (\rho V\rho) & = \Delta A(\rho), \\ 
\left(\grad A(\rho) + \rho V\rho\right) \cdot \hat n|_{\partial \Omega} & = 0, 
\end{align*}
with $A:\Real \rightarrow [0,\infty)$ non-decreasing and measurable (possibly zero). 
Suppose there exists two weak solutions $\rho_1(t),\rho_2(t)$ defined on $[0,T]$ which satisfy $\rho_i(t) \in L^{2}_t(0,T;BMO(\Omega)) \cap C_t([0,T];L^{p_0}_x(\Omega))$ where $p_0 \geq 1$. 
Then if $\rho_1(0) = \rho_2(0)$, then $\rho_1(t) = \rho_2(t)$ on $[0,T]$. 
\end{theorem}
\begin{remark} 
See also \cite{BertozziSlepcev10} for a similar method in a periodic setting.  
\end{remark} 
\begin{proof}
Define $w := \rho_1 - \rho_2$, $\phi$ as the unique mean-zero solution to the Neumann problem 
\begin{align*} 
-\Delta \phi & = w \\ 
\grad \phi \cdot \hat n|_{\partial \Omega} & = 0, 
\end{align*}
and note $\norm{w}_{H^{-1}(\Omega)} := \norm{\grad \phi}_{2}$. 
Similar to above, by Lemma \ref{lem:BMO_linear_p} and standard elliptic regularity, $\phi \in L^2_tL_x^\infty$ and $\grad \phi \in L_t^2L_x^2$. Note that since $\Omega$ is bounded, this holds in two dimensions without requiring any additional assumptions or the use of Lemma \ref{lem:Hardy}. 
Similar to the proof of Theorem \ref{thm:Hdot1Inviscid}, we may use $\phi$ as a test function in the definition of weak solution and compute the time evolution of $\norm{\grad \phi(t)}_2$. 
Using the homogeneous Neumann data on $\phi$ and no flux conditions on $\rho_i$, we get as in Theorem \ref{thm:Hdot1Inviscid}, 
\begin{align*}
\frac{1}{2}\frac{d}{dt}\int \abs{\grad \phi}^2 dx 
& = -\int \left(A(\rho_1) - A(\rho_2)\right)(\rho_1 - \rho_2) dx +\int \grad \phi \cdot wV\rho_1 dx + \int \grad \phi \cdot \rho_2(V\rho_1 - V\rho_2) dx \\ & := T1 + T2 + T3.   
\end{align*}
The $T1$ and $T3$ terms may be treated as in Theorem \ref{thm:Hdot1Inviscid} (although checking condition \textbf{B3} may be harder than in the $\Real^d$ case). 
Condition \textbf{B1} arises in the treatment of $T2$. Indeed, using the boundary condition on $\phi$ and \textbf{B3} in the last line, 
\begin{align*} 
T2 & = \sum_{ij} \int \partial_i \phi \partial_{ij} \phi (V\rho_1)_j + \partial_i\phi \partial_j\phi \partial_i (V\rho_1)_j dx \\
& = \sum_{ij} \int \partial_j\left(\frac{1}{2} \abs{\partial_i\phi}^2 \right) (V\rho_1)_j + \partial_i\phi \partial_j\phi \partial_i (V\rho_1)_j dx \\
 & = \sum_{ij} \int_{\partial \Omega} \left(\frac{1}{2} \abs{\partial_i\phi}^2 \right) (V\rho_1)_j \hat n_j dS - \sum_{ij} \int \left(\frac{1}{2} \abs{\partial_i\phi}^2 \right) \partial_j(V\rho_1)_j - \partial_i\phi \partial_j\phi \partial_i (V\rho_1)_j dx \\
& \lesssim \int \abs{\grad V\rho_1}\abs{\grad \phi}^2 dx. 
\end{align*}
From here the proof proceeds as in Theorem \ref{thm:Hdot1Inviscid}. 
\end{proof}

\section*{Appendix: Proof of Lemma \ref{lem:BMO_logLip}}

\begin{proof}
We assume $v \in \mathcal{S}(\Real^d)$ and use density to extend to more general functions. 
Let $x,y\in \Real^d$ and let $r = \abs{x - y}$. Define $W := B(x,r) \cap B(y,r)$. By the triangle inequality, 
\begin{equation*}
\abs{v(x) - v(y)} \leq \avint{W} \abs{v(x) - v(z)} dz + \avint{W} \abs{v(y) - v(z)} dz.
\end{equation*}
Moreover, 
\begin{equation*}
\avint{W} \abs{v(x) - v(z)}dz \leq \frac{C}{\abs{B(x,r)}}\int \abs{v(x)-v(z)}dz, 
\end{equation*}
where $C$ depends on the ratio between $\abs{B(x,r)}$ and $\abs{W}$, which is fixed in $r$. 
A standard fundamental theorem of calculus argument (see e.g. \cite{Evans} pg. 267) implies
\begin{equation}
\frac{1}{\abs{B(x,r)}}\int \abs{v(x)-v(z)}dz \lesssim \int_{B(x,r)} \frac{\abs{\grad v(z)}}{\abs{x-z}^{d-1}}dz. \label{ineq:FTC}. 
\end{equation}   
We are concerned with controlling this integral for $r \ll 1$. 
Without loss of generality we may assume $x = 0$. 
For notational simplicity, define $\grad v(y) = f(y)$ and $B_k := B(0,2^{-k})$.
Now, 
\begin{align*}
\int_{B(0,r)} \frac{\abs{f(z)}}{\abs{z}^{d-1}} dz \leq \int_{B(0,r)} \frac{\abs{f(z) - f_{B_1}}}{\abs{z}^{d-1}} dz + f_{B_1}C(d)r. 
\end{align*}
Let us now focus on the first term, as the latter term is uniformly $\mathcal{O}(r)$ by the assumption on the local averages of $\grad v$. 
By $f \in BMO$, for all $k \in \Naturals$ we have, 
\begin{equation*}
\int_{B_k}\abs{f - f_{B_k}}dx \leq c2^{-dk}\norm{f}_{BMO}. 
\end{equation*}
Moreover, as $f\in BMO$, we have
\[|f_{B_{k}}-f_{B_{k+1}}|\lesssim_{d} ||f||_{BMO},\] 
which implies $\abs{f_{B_k} - f_{B_1}} \leq ck\norm{f}_{BMO}$.
We come to the main estimate, which breaks the integral into successive length-scales,
\begin{align*}
\int_{\abs{z} < r} \frac{\abs{f - f_{B_1}}}{\abs{z}^{d-1}} dz & \sim \sum_{k \geq\log_2 r}^{\infty} \int_{\abs{z}\sim 2^{-k}}\frac{\abs{f - f_{B_1}}}{2^{-k(d-1)}} dz 
 \leq \sum_{k \geq\log_2 r}^{\infty} \int_{\abs{z}\leq 2^{-k}}\frac{\abs{f - f_{B_1}}}{2^{-k(d-1)}} dz \\
& \lesssim_{d} \norm{f}_{BMO}\sum_{k \geq\log_2 r}^{\infty} \frac{k2^{-kd}}{2^{-k(d-1)}} = \norm{f}_{BMO}\sum_{k \geq \log_2 r}k2^{-k}.  
\end{align*}
An elementary computation shows
\begin{equation*}
\sum_{k \geq \log_2 r} kx^{k-1} =  \frac{d}{dx}\sum_{k \geq \log_2 r} x^{k} = \frac{d}{dx}\frac{x^{\log_2 r}}{1 -x}. 
\end{equation*}
This finally implies, 
\begin{equation*}
\int_{B(x,r)}\frac{\grad v(z)}{\abs{z-x}^{n-1}} dz \lesssim \norm{\grad v}_{BMO}(r\abs{\log r} - r), 
\end{equation*}
for $r$ sufficiently small depending on $\avint{B(x,1)} \abs{\grad v} dx$, which is uniformly bounded. 
Hence, for $r$ sufficiently small, 
\begin{align*}
\abs{v(x)} & \leq \avint{B(x,r)}\abs{v(x) - v(y)} dy + \avint{B(x,r)}\abs{v(y)} dy \\ 
& \leq \int_{B(x,r)}\frac{\grad v(z)}{\abs{z-x}^{n-1}} dz + Cr^{d/p - d}\norm{v}_{L^p(B(x,r))} \\ 
& \lesssim \norm{\grad v}_{BMO}(r\abs{\log r} - r) + Cr^{d(p-1)/p - d}\norm{v}_p.
\end{align*}
This completes the lemma. 
\end{proof}

\section*{Acknowledgments} 
J. Bedrossian would like to thank Andrea Bertozzi, Inwon Kim, Eric Carlen and Nancy Rodr\'iguez for helpful discussions. They would also like to thank the anonymous referee, whose suggestions led to substantial improvements in the content of this paper, particularly the results of section 5.
This work was in part supported by NSF grant DMS-1103765 and NSF grant DMS-0907931. Part of this work was done while J. Azzam was visiting the Mathematical Sciences Research Institute.

\bibliographystyle{plain} \bibliography{nonlocal_eqns,dispersive}

\begin{thebibliography}{10}

\bibitem{AmbrosioIM04}
L.~Ambrosio.
\newblock Transport equation and {Cauchy} problem for {BV} vector fields.
\newblock {\em Invent. Math.}, 158:227--260, 2004.

\bibitem{MasmoudiBedrossian12}
J.~Bedrossian and N.~Masmoudi.
\newblock Existence, uniqueness and {Lipschitz} dependence for
  {Patlak-Keller-Segel} and {Navier-Stokes} in {R2} with measure-valued initial
  data.
\newblock {\em {arXiv}:1205.1551}, 2012.

\bibitem{BR11}
J.~Bedrossian and N.~Rodr\'iguez.
\newblock Inhomogenous {Patlak-Keller-Segel} models and aggregation equations
  with nonlinear diffusion in $\mathbb{R}^d$.
\newblock {\em {arXiv}:1108.5167, To appear in Disc. Cont. Dyn. Sys. A}, 2012.

\bibitem{BRB10}
J.~Bedrossian, N.~Rodr\'iguez, and A.L. Bertozzi.
\newblock Local and global well-posedness for aggregation equations and
  {Patlak-Keller-Segel} models with degenerate diffusion.
\newblock {\em Nonlinearity}, 24(6):1683--1714, 2011.

\bibitem{BertozziBrandman10}
A.L. Bertozzi and J.~Brandman.
\newblock Finite-time blow-up of {$L^\infty$}-weak solutions of an aggregation
  equation.
\newblock {\em Comm. Math. Sci.}, 8(1):45--65, 2010.

\bibitem{BertozziLaurentRosado10}
A.L. Bertozzi, T.~Laurent, and J.~Rosado.
\newblock {$L^p$} theory for the multidimensional aggregation equation.
\newblock {\em Comm. Pure. Appl. Math.}, 64(1), 2010.

\bibitem{BertozziSlepcev10}
A.L. Bertozzi and D.~Slep\v{c}ev.
\newblock Existence and uniqueness of solutions to an aggregation equation with
  degenerate diffusion.
\newblock {\em Comm. Pure. Appl. Anal.}, 9(6):1617--1637, 2010.

\bibitem{Blanchet09}
A.~Blanchet, J.A. Carrillo, and P.~Lauren\c{c}ot.
\newblock Critical mass for a {Patlak-Keller-Segel} model with degenerate
  diffusion in higher dimensions.
\newblock {\em Calc. Var.}, 35:133--168, 2009.

\bibitem{BlanchetEJDE06}
A.~Blanchet, J.~Dolbeault, and B.~Perthame.
\newblock Two-dimensional {Keller-Segel} model: Optimal critical mass and
  qualitative properties of the solutions.
\newblock {\em E. J. Diff. Eqn}, 2006(44):1--32, 2006.

\bibitem{CaffarelliVasseur06}
L.~Caffarelli and A.~Vasseur.
\newblock Drift diffusion equations with fractional diffusion and the
  quasi-geostrophic equation.
\newblock {\em arXiv:0608447}, 2006.

\bibitem{CarrilloRosado10}
J.A. Carrillo and J.~Rosado.
\newblock Uniqueness of bounded solutions to aggregation equations by optimal
  transport methods.
\newblock {\em Proc. 5th Euro. Congress of Math. Amsterdam}, 2008.

\bibitem{Constantin94}
P.~Constantin.
\newblock Active scalars and the {Euler} equations.
\newblock {\em Tatra Mountains Math. Publ.}, 4:25--38, 1994.

\bibitem{ConstantinCordobaWu01}
P.~Constantin, D.~Cordoba, and J.~Wu.
\newblock On the critical dissipative quasi-geostrophic equation.
\newblock {\em Indiana U. Math. J.}, 50:97--107, 2001.

\bibitem{ConstantinIyerWu08}
P.~Constantin, G.~Iyer, and J.~Wu.
\newblock Global regularity for a modified critical dissipative
  quasi-geostrophic equation.
\newblock {\em Indiana U. Math. J.}, 57(6):2681--2692, 2008.

\bibitem{DeLellisSzekelyhidi08}
C.~De~Lellis and L.~Sz{\'e}kelyhidi.
\newblock On admissibility criteria for weak solutions of the euler equations.
\newblock {\em Arch. Rat. Mech. Anal.}, 195(1):225--260, 2010.

\bibitem{DiPernaLions89}
R.J. DiPerna and P.L. Lions.
\newblock Ordinary differential equations, transport theory and {Sobolev}
  spaces.
\newblock {\em Invent. Math.}, 98:511--547, 1989.

\bibitem{Evans}
L.C. Evans.
\newblock {\em Partial Differential Equations}, volume~19 of {\em Grad. Stud.
  Math.}
\newblock American Mathematical Society, 1998.

\bibitem{GallagherGallay05}
I.~Gallagher and T.~Gallay.
\newblock Uniqueness for the two-dimensional {Navier-Stokes} equation with
  measure as initial vorticity.
\newblock {\em Math. Ann.}, 332:287--327, 2005.

\bibitem{Jones80}
P.W. Jones.
\newblock Extension theorems for bmo.
\newblock {\em Indiana Univ. Math. J}, 29(1):41--66, 1980.

\bibitem{Kelliher11}
J.~Kelliher.
\newblock On the flow map for {2D Euler} equations with unbounded vorticity.
\newblock {\em Nonlinearity}, 24(9), 2011.

\bibitem{Kiselev10}
A.~Kiselev.
\newblock Regularity and blow up for active scalars.
\newblock {\em arXiv:1009.0540}, 2010.

\bibitem{KiselevNazarov07}
A.~Kiselev, F.~Nazarov, and A.~Volberg.
\newblock Global well-posedness of the critical {2D} dissipative
  quasi-geostrophic equation.
\newblock {\em Invent. Math.}, 167:445--453, 2007.

\bibitem{KozonoWadade08}
H.~Kozono and H.~Wadade.
\newblock Remarks on {Gagliardo-Nirenberg} type inequality with critical
  {Sobolev} space and {BMO}.
\newblock {\em Math. Z.}, 259:935--950, 2008.

\bibitem{Laurent07}
T.~Laurent.
\newblock Local and global existence for an aggregation equation.
\newblock {\em Comm. Part. Diff. Eqn.}, 32:1941--1964, 2007.

\bibitem{LoeperSG06}
G.~Loeper.
\newblock A fully nonlinear version of the incompressible {Euler} equations:
  the semigeostrophic system.
\newblock {\em SIAM J. Math. Anal.}, 38(3):795--823, 2006.

\bibitem{LoeperVP06}
G.~Loeper.
\newblock Uniqueness of the solution to the {Vlasov-Poisson} system with
  bounded density.
\newblock {\em J. Math. Pures Appl.}, 86:68--79, 2006.

\bibitem{MajdaBertozzi}
A.~Majda and A.~L. Bertozzi.
\newblock {\em Vorticity and Incompressible Flow}.
\newblock Cambridge University Press, 2002.

\bibitem{mcmullen1996renormalization}
Curtis~T. McMullen.
\newblock {\em Renormalization and 3-manifolds which fiber over the circle},
  volume 142 of {\em Annals of Mathematics Studies}.
\newblock Princeton University Press, Princeton, NJ, 1996.

\bibitem{Robert97}
R.~Robert.
\newblock Unicite\'e de la solution faible \`a support compact de l'\'equation
  de {Vlasov-Poisson}.
\newblock {\em C.R. Acad. Sci. Paris, S\'{e}r. I Math}, 324(8):873--877, 1997.

\bibitem{Rusin11}
W.~Rusin.
\newblock Logarithmic spikes of gradients and uniqueness of weak solutions to a
  class of active scalar equations.
\newblock {\em Preprint, {\textup{arXiv:1106.2778}}}, 2011.

\bibitem{Scheffer93}
V.~Scheffer.
\newblock An inviscid flow with compact support in space-time.
\newblock {\em J. Geom. Anal.}, 3(4), 1993.

\bibitem{Shnirelman97}
A.~Shnirelman.
\newblock On the nonuniqueness of weak solution of the {Euler} equation.
\newblock {\em Comm. Pure Appl. Math.}, 50(12), 1997.

\bibitem{LittleStein}
E.~Stein.
\newblock {\em Singualar Integrals and Differentiability Properties of
  Functions}.
\newblock Princeton University Press, 1970.

\bibitem{BigStein}
E.~Stein.
\newblock {\em {Harmonic Analysis: Real-Variable Methods, Orthogonality, and
  Oscillatory Integrals}}.
\newblock Princeton University Press, 1993.

\bibitem{Vaisala88}
J.~V{\"a}is{\"a}l{\"a}.
\newblock Uniform domains.
\newblock {\em T{\^o}hoku Mathematical Journal}, 40(1):101--118, 1988.

\bibitem{VazquezPME}
J.L. V\'{a}zquez.
\newblock {\em The Porous Medium Equations}.
\newblock Clarendon Press, Oxford, 2007.

\bibitem{Vishik99}
M.~Vishik.
\newblock Incompressible flows of an ideal fluid with vorticity in borderline
  spaces of {Besov} type.
\newblock {\em Ann. Scient. \'Ec. Norm. Sup.}, 32:769--812, 1999.

\bibitem{Wu05}
J.~Wu.
\newblock Solutions of the {2D} quasi-geostrophic equation in {H\"older}
  spaces.
\newblock {\em Nonlin. Anal.}, 62:579--594, 2005.

\bibitem{Yudovich63}
V.I. Yudovich.
\newblock Non-stationary flow of an ideal incompressible liquid.
\newblock {\em Zh. Vychisl. Mat. Fiz.}, 3(6):1032--1066, 1963.

\bibitem{Yudovich95}
V.I. Yudovich.
\newblock Uniqueness theorem for the basic nonstationary problem in the
  dynamics of an ideal incompressible fluid.
\newblock {\em Math. Research. Letters}, 2:27--38, 1995.

\end{thebibliography}

\end{document}